\documentclass[runningheads,envcountsame]{llncs}
\usepackage{microtype} 
\usepackage{xcolor}
\usepackage{thm-restate}
\usepackage{etoolbox}
\usepackage{amsmath}

\newcommand{\prooflink}[1]{}
\newcommand{\statlink}[1]{}
\newcommand{\mylink}{}

\usepackage{dsfont}
\newcommand{\N}{\ensuremath{\mathds{N}}}
\newcommand{\R}{\ensuremath{\mathds{R}}}
\usepackage{graphicx}
\usepackage{enumitem}
\usepackage{latexsym,amsfonts,amssymb,wrapfig,subfigure}

\makeatletter
\if@envcntsame
\let\c@case\relax

\@addtoreset{case}{theorem}
\fi
\makeatother

\usepackage{cleveref}
\crefformat{observation}{Observation #2#1#3}
\crefformat{@theorem}{Theorem #2#1#3}
\crefformat{theorem}{Theorem #2#1#3}
\crefformat{lemma}{Lemma #2#1#3}
\crefformat{claim}{Claim #2#1#3}
\crefformat{section}{Section #2#1#3}
\crefformat{figure}{Figure #2#1#3}
\crefformat{case}{Case #2#1#3}

\graphicspath{{./figures/}}

\newcommand{\rightharpoonupfill}{{$\mathsurround=0pt\mathord-\mkern-6mu%
  \cleaders\hbox{$\mkern-2mu\mathord-\mkern-2mu$}\hfill
  \mkern-6mu\mathord\rightharpoonup$}}
\newbox\myrhpmbox\setbox\myrhpmbox=\hbox{$\rightharpoonup$}
\newbox\myrhpobox\newbox\myrhptbox
\newcommand{\half}[1]{\setbox\myrhpobox=\hbox{\ensuremath{#1}}\ifdim\wd\myrhpmbox>\wd\myrhpobox\setbox\myrhptbox=\hbox to \wd\myrhpmbox{\hfil\box\myrhpobox\hfil}\setbox\myrhpobox=\box\myrhptbox\setbox\myrhptbox=\copy\myrhpmbox\else\setbox\myrhptbox=\hbox to \wd\myrhpobox{\rightharpoonupfill}\fi\vbox{\baselineskip=0pt\box\myrhptbox\kern-0.75\ht\myrhpmbox\box\myrhpobox}}
    
\begin{document}
\title{The Number of Edges in Maximal 2-planar Graphs\thanks{
    Both the authors are supported by the Swiss National Science Foundation within the collaborative DACH project \emph{Arrangements and Drawings} as SNSF Project 200021E-171681.
    This work (without appendix) is available at the 39th International Symposium on Computational Geometry
(SoCG 2023).
}}

%
%
\author{Michael Hoffmann\inst{1}\orcidID{0000-0001-5307-7106} \and
Meghana~M.~Reddy\inst{1}\orcidID{0000-0001-9185-1246}}
\authorrunning{M. Hoffmann, M. M. Reddy}
%
\institute{Department of Computer Science, ETH Z\"urich, Switzerland \email{hoffmann,meghana.mreddy@inf.ethz.ch}}
\maketitle              

\begin{abstract} 
  A graph is \emph{$2$-planar} if it has local crossing number two,
  that is, it can be drawn in the plane such that every edge has at
  most two crossings.  A graph is \emph{maximal $2$-planar} if no edge
  can be added such that the resulting graph remains $2$-planar.  A
  $2$-planar graph on~$n$ vertices has at most~$5n-10$ edges, and some
  (maximal) $2$-planar graphs---referred to as \emph{optimal
    $2$-planar}---achieve this bound. However, in strong contrast to
  maximal planar graphs, a maximal $2$-planar graph may have fewer
  than the maximum possible number of edges. In this paper, we
  determine the minimum edge density of maximal $2$-planar graphs by
  proving that every maximal $2$-planar graph on $n\ge 5$~vertices has
  at least~$2n$ edges. We also show that this bound is tight, up to an
  additive constant. The lower bound is based on an analysis of the
  degree distribution in specific classes of drawings of the
  graph. The upper bound construction is verified by carefully
  exploring the space of admissible drawings using computer support.

\end{abstract}
\section{Introduction}

Maximal planar graphs a.k.a.~(combinatorial) triangulations are a
rather important and well-studied class of graphs with a number of
nice and useful properties. To begin with, the number of edges is
uniquely determined by the number of vertices, as every maximal planar
graph on~$n\ge 3$ vertices has~$3n-6$ edges. It is natural to wonder
if a similar statement can be made for the various families of
near-planar graphs, which have received considerable attention over
the past decade; see,
e.g.~\cite{dlm-sgdbp-20,ht-bpg-20}. 

In this paper we focus on $k$-planar graphs, specifically
for~$k=2$. These are graphs with local crossing number at most~$k$,
that is, they admit a drawing in~$\R^2$ where every edge has at
most~$k$ crossings. The class of~$1$-planar graphs was introduced by
Ringel~\cite{Ringel65} in the context of vertex-face colorings of
planar graphs. Later, Pach and T{\'o}th~\cite{PachT97} used upper
bounds on the number of edges in~$k$-planar graphs to derive an
improved version of the Crossing Lemma, which gives a lower bound on
the crossing number of a simple (no loops or multi-edges) graph in
terms of its number of vertices and edges. The class of~$k$-planar
graphs is not closed under edge contractions and already for~$k=1$
there are infinitely many minimal non-$1$-planar graphs, as shown by
Korzhik~\cite{k-mn1pg-08}.

The maximum number of edges in a~$k$-planar graph on~$n$ vertices
increases with~$k$, but the exact dependency is not known. A general
upper bound of~$O(\sqrt{k}n)$ is known due to Ackerman and Pach and
T{\'oth}~\cite{ackerman2019topological,PachT97} for graphs that admit
a \emph{simple} $k$-plane drawing, that is, a drawing where every pair
of edges has at most one common point. Only for small~$k$ we have
tight bounds. A $1$-planar graph on~$n$~vertices has at most~$4n-8$
edges and there are infinitely many \emph{optimal}~$1$-planar graphs
that achieve this bound, as shown by Bodendiek, Schumacher, and
Wagner~\cite{bsw-bsgr-83}. 
A $2$-planar graph on~$n$~vertices has at most~$5n-10$ edges and there
are infinitely many \emph{optimal}~$2$-planar graphs that achieve this
bound, as shown by Pach and T{\'o}th~\cite{PachT97}. In fact, there
are complete characterizations, for optimal $1$-planar graphs by
Suzuki~\cite{s-o1pgtos-10} and for optimal $2$-planar graphs by Bekos,
Kaufmann, and Raftopoulou~\cite{bekos_optimal2planargraphs}.

Much less is known about \emph{maximal} $k$-planar graphs, that is,
graphs for which adding any edge results in a graph that is
not~$k$-planar anymore. In contrast to planar graphs, where maximal
and optimal coincide, it is easy to find examples of maximal
$k$-planar graphs that are not optimal; a trivial example is the
complete graph~$K_5$. In fact, the difference between maximal and
optimal can be quite large for~$k$-planar graphs, even---perhaps
counterintuitively---maximal~$k$-planar graphs for~$k\ge1$ may have
fewer edges than maximal planar graphs on the same number of vertices.
Hud{\'a}k, Madaras, and Suzuki~\cite{hms-pm1pg-12} describe an
infinite family of maximal $1$-planar graphs with only
${8n}/{3}+O(1) \approx 2.667n$ edges. An improved construction with
${45n}/{17}+O(1) \approx 2.647n$ edges was given by Brandenburg,
Eppstein, Glei{\ss}ner, Goodrich, Hanauer, and
Reislhuber~\cite{brandenburg_1planar} who also established a lower
bound by showing that every maximal $1$-planar graph has at least
${28n}/{13}-O(1) \approx 2.153n$ edges. Later, this lower bound was
improved to~$20n/9\approx 2.22n$ by Bar\'{a}t and
T\'{o}th~\cite{BaratT18}.

Maximal $2$-planar graphs were studied by Auer, Brandenburg,
Glei{\ss}ner, and Hanauer~\cite{brandenburg_2planar} who constructed
an infinite family of maximal $2$-planar graphs with~$n$ vertices
and~${387n}/{147}+O(1) \approx 2.63n$ edges.\footnote{Maximality is
  proven via uniqueness of the $2$-plane drawing of the
  graph. However, there is no explicit proof of the uniqueness in this
  short abstract.}  We are not aware of any nontrivial lower bounds on
the number of edges in maximal $k$-planar graphs, for~$k\ge 2$.

\paragraph*{Results.} In this paper, we give tight bounds on the minimum number of edges in maximal~$2$-planar graphs, up to an additive constant.

\begin{restatable}{theorem}{thmmain}\label{thm:main}
  Every maximal $2$-planar graph on $n\ge 5$~vertices has at least~$2n$ edges.
\end{restatable}

\begin{restatable}{theorem}{thmupper}\label{thm:upper}
  There exists a constant~$c\in\N$ such that for every~$n\in\N$ there
  exists a maximal $2$-planar graph on~$n$ vertices with at
  most~$2n+c$ edges.
\end{restatable}

\paragraph*{Related work.} Maximality has also been studied for
drawings of simple graphs. Let~$\mathcal{D}$ be a class of drawings. A
drawing~$D\in\mathcal{D}$ is \emph{saturated} if no edge can be added
to~$D$ so that the resulting drawing is still in~$\mathcal{D}$. For
the class of simple drawings, Kyn\v{c}l, Pach, Radoi\v{c}i{\'c} and
T{\'o}th~\cite{KYNCL2015295} showed that every saturated drawing
on~$n$ vertices has at least~$1.5n$ edges and there exist saturated
drawings with no more than~$17.5n$ edges. The upper bound was improved
to~$7n$ by Hajnal, Igamberdiev, Rote and Schulz~\cite{hirs-sstgfe-18}.
Chaplick, Klute, Parada, Rollin, and Ueckerdt~\cite{ckpru-eskd-21}
studied saturated $k$-plane drawings, for~$k\ge 4$, and obtained tight
bounds linear in~$n$, where the constant depends on~$k$, for various
types of crossing restrictions. For the class of $1$-plane drawings,
Brandenburg, Eppstein, Glei{\ss}ner, Goodrich, Hanauer, and
Reislhuber~\cite{brandenburg_1planar} showed that there exist
saturated drawings with no more than~$7n/3 + O(1)\approx 2.33n$ edges.
On the lower bound side, the abovementioned bound
of~$20n/9\approx 2.22n$ edges by Bar\'{a}t and
T\'{o}th~\cite{BaratT18} actually holds for saturated $1$-plane
drawings. For the class of $2$-plane drawings, Auer, Brandenburg,
Glei{\ss}ner, and Hanauer~\cite{brandenburg_2planar} describe
saturated drawings with no more than~$4n/3 + O(1)\approx 1.33n$ edges,
and Bar\'{a}t and T\'{o}th~\cite{bt-s2dfe-21} show that every
saturated $2$-plane drawing on~$n$ vertices has at least~$n-1$ edges.

Although the general spirit is similar, saturated drawings are quite
different from maximal abstract graphs. To obtain a sparse saturated
drawing, one can choose both the graph and the drawing, whereas for
sparse maximal graphs one can choose the graph only and needs to get a
handle on all possible drawings. Universal lower bounds for saturated
drawings carry over to the maximal graph setting, and existential
upper bounds for maximal graphs carry over to saturated drawings. But
bounds obtained in this fashion are far from tight usually; compare,
for instance, the range of between~$n$ and~$4n/3$ edges for saturated
$2$-plane drawings to our bound of~$2n$ edges for maximal $2$-planar
graphs.


\section{Preliminaries}

A \emph{drawing} of a graph~$G=(V,E)$ is a map~$\gamma:G\to\R^2$ that maps each vertex~$v\in V$ to a point~$\gamma(v)\in\R^2$ and each each edge~$uv\in E$ to a simple (injective) curve~$\gamma(uv)$ with endpoints~$\gamma(u)$ and~$\gamma(v)$, subject to the following conditions: (1)~$\gamma$ is injective on~$V$; (2)~for all~$uv\in E$ we have~$\gamma(uv)\cap\gamma(V)=\{\gamma(u),\gamma(v)\}$; and (3)~for each pair~$e_0,e_1\in E$ with~$e_0\ne e_1$ the curves~$\gamma(e_0)$ and~$\gamma(e_1)$ have at most finitely many intersections, and each such intersection is either a common endpoint or a proper, transversal \emph{crossing} (that is, no touching points between these curves). The connected components of~$\R^2\setminus\gamma(G)$ are the \emph{faces} of~$\gamma$. The \emph{boundary} of a face~$f$ is denoted by~$\partial f$.

To avoid notational clutter we will often identify vertices and edges with their geometric representations in a given drawing. A drawing is \emph{simple} if every pair of edges has at most one common point. A drawing is \emph{$k$-plane}, for~$k\in\N$, if every edge has at most~$k$ crossings. A graph is \emph{$k$-planar} if it admits a~$k$-plane drawing. A graph is \emph{maximal $k$-planar} if no edge can be added to it so that the resulting graph is still~$k$-planar.

To analyze a~$k$-planar graph one often analyzes one of its~$k$-plane drawings. It is, therefore, useful to impose additional restrictions on this drawing if possible. One such restriction is to consider a \emph{crossing-minimal}~$k$-plane drawing, that is, a drawing that minimizes the total number of edge crossings among all~$k$-plane drawings of the graph. For small~$k$, such a drawing is always simple; for~$k\ge 4$ this is not the case in general~\cite[Footnote~112]{schaefer2012graph}.

\begin{lemma}[Pach, Radoi{\v{c}}i{\'{c}}, Tardos, and T{\'{o}}th~{\cite[Lemma~1.1]{PachRTT06}}]\label{thm:2planar_simple}
  For~$k\le 3$, every crossing-minimal $k$-plane drawing is simple.
\end{lemma}

In figures, we use the following convention to depict edges: Uncrossed
edges are shown green, singly crossed edges are shown purple, doubly
crossed edges are shown blue, and edges for which the number of
crossings is undetermined are shown black.

\paragraph*{Connectivity.} Next let us collect some basic
properties of maximal~$k$-planar graphs and their drawings. Some of
these may be folklore, but for completeness we include the (simple)
proofs in \cref{sec:basic}. (In the PDF, the blue triangle pointing up
links to the statement in the main text and the red triangle pointing
down links to the proof in the appendix.)

\begin{restatable}{lemma}{UncrossedEdge}\label{lem:uncrossed_edge}
  \renewcommand{\mylink}{\statlink{lem:uncrossed_edge}\prooflink{PUncrossedEdge}}
  Let $D$ be a crossing-minimal $k$-plane drawing of a maximal
  $k$-planar graph $G$, and let~$u$ and~$v$ be two vertices that lie
  on (the boundary of) a common face in~$D$. Then $uv$ is an edge
  of~$G$ and it is uncrossed in~$D$.
\end{restatable}

\begin{restatable}{lemma}{VertexIncidentToUncrossedEdge}\label{lem:vertex_incident_to_uncrossededge}
  \renewcommand{\mylink}{\statlink{lem:vertex_incident_to_uncrossededge}\prooflink{PVertexIncidentToUncrossedEdge}}
  Let $D$ be a crossing-minimal $k$-plane drawing of a maximal $k$-planar graph on~$n$ vertices, for~$k\le 2\le n$. Then every vertex is incident to an 
  uncrossed edge in~$D$.
\end{restatable}

\begin{restatable}{lemma}{TwoConnected}\label{thm:2connected}
  \renewcommand{\mylink}{\statlink{thm:2connected}\prooflink{PTwoConnected}}
  For~$k\le 2$, every maximal $k$-planar graph on~$n\ge 3$ vertices is $2$-connected.
\end{restatable}

\section{The Lower Bound}
\label{sec:low_deg_vertices}

In this section we develop our lower bound on the edge density of
maximal $2$-planar graphs by analyzing the distribution of vertex
degrees. As we aim for a lower bound of~$2n$ edges, we want to show
that the average vertex degree 
is at least four. Then, the density bound follows by the handshaking
lemma.  However, maximal $2$-planar graphs may contain vertices of
degree less than four. By \cref{thm:2connected} we know that the
degree of every vertex is at least two. But degree-two vertices,
so-called \emph{hermits}, may exist, as well as vertices of degree
three.

In order to lower bound the average degree by four, we employ a
charging scheme where we argue that every \emph{low-degree} vertex,
that is, every vertex of degree two and three claims a certain number
of halfedges at an adjacent \emph{high-degree} vertex, that is, a
vertex of degree at least five. Claims are exclusive, that is, every
halfedge at a high-degree vertex can be claimed at most once. We use
the term \emph{halfedge} because the claim is not on the whole edge
but rather on its incidence to one of its high-degree endpoints. The
incidence at the other endpoint may or may not be claimed
independently (by another vertex). For an edge~$uv$ we denote
by~$\half{uv}$ the corresponding halfedge at~$v$ and by~$\half{vu}$
the corresponding halfedge at~$u$. A halfedge~$\half{uv}$ inherits the
properties of its \emph{underlying} edge~$uv$, such as being crossed
or uncrossed in a particular drawing.
Vertices of degree four have a special role, as they are neither low--
nor high-degree. However, a vertex of degree four that is adjacent to
a hermit is treated like a low-degree vertex. More precisely, our
charging scheme works as follows:
\begin{enumerate}[label=(C\arabic*),left=\labelsep]
\item\label{c:1} Every hermit claims two halfedges at each high-degree
  neighbor.
\item\label{c:2} Every degree-three vertex claims three halfedges at
  some high-degree neighbor.
\item\label{c:3} Every degree four vertex that is adjacent to a
  hermit~$h$ claims two halfedges at some
  neighbor~$v$ of degree $\geq 6$. Further, the vertices~$h$ and~$v$ are adjacent, 
  so~$h$ also claims two halfedges at~$v$
  by~\ref{c:1}. If~$\deg(v)=6$, then~$v$ is adjacent to exactly one
  hermit.
\item\label{c:4} At most one vertex claims (one or more) halfedges at a degree five
  vertex.
\end{enumerate}

The remainder of this section is organized as follows. First, we
present the proof of \cref{thm:main} in \cref{sec:LP}. Then we prove
the validity of our charging scheme along with some useful properties
of low-degree vertices in
\cref{sec:remarks}--\ref{sec:charge}. Specifically, we will use the
following statements in the proof of \cref{thm:main} below.

\begin{restatable}{lemma}{EfficientHermitNeighbors}\label{lem:efficient_hermit_neighbors}
  \renewcommand{\mylink}{\statlink{lem:efficient_hermit_neighbors}\prooflink{PEfficientHermitNeighbors}}
  Let~$G$ be a maximal $2$-planar graph on~$n\ge 5$ vertices, let~$h$
  be a hermit, and let~$x,y$ be the neighbors of~$h$ in~$G$. Then we
  have~$\deg(x)\ge 4$ and~$\deg(y)\ge 4$.
\end{restatable}

\begin{restatable}{lemma}{HermitDeg}\label{prop:hermit_deg}
  \renewcommand{\mylink}{\statlink{prop:hermit_deg}\prooflink{PHermitDeg}}
  Let~$G$ be a maximal $2$-planar graph on~$n\ge 5$ vertices. Then a
  vertex of degree~$i$ in~$G$ is adjacent to at
  most~$\lfloor i/3\rfloor$ hermits.
\end{restatable}

\subsection{Proof of \cref{thm:main}}\label{sec:LP}

Let~$G$ be a maximal $2$-planar graph on~$n\ge 5$ vertices, and
let~$m$ denote the number of edges in~$G$. We denote by~$v_i$ the
number of vertices of degree~$i$ in~$G$. By \cref{thm:2connected} we
know that~$G$ is $2$-connected and, therefore, we
have~$v_0 = v_1 = 0$. Thus, we have
\begin{equation}\label{eq:n}
n = \sum_{i=2}^{n-1} v_i
\;\;\text{and by the handshaking lemma}\;\;
2m = \sum_{i=2}^{n-1} i \cdot v_i.
\end{equation}

Vertices of degree four or higher can be adjacent to
hermits. Let~$v_i^{\mathrm{h}j}$ denote the number of vertices of
degree~$i$ incident to~$j$ hermits in~$G$. By~\cref{prop:hermit_deg}
we have
\begin{equation}\label{eq:degsum}
  v_i = \sum_{j=0}^{\lfloor i/3 \rfloor} v_i^{\mathrm{h}j}
  \hspace{1cm} \text{for all } i \ge 3.
\end{equation}

By \cref{lem:efficient_hermit_neighbors}, both neighbors of a hermit
have degree at least four.
Thus, double counting the edges between hermits and their neighbors we
obtain
\begin{align}\label{eq:deg2}
  2 v_2 \leq v_4^{\mathrm{h}1} + v_5^{\mathrm{h}1} +
  v_6^{\mathrm{h}1} + 2 v_6^{\mathrm{h}2} +
  v_7^{\mathrm{h}1} + 2 v_7^{\mathrm{h}2}
  + 2v_8 + v_9^{\mathrm{h}1} + 2 v_9^{\mathrm{h}2} +  \nonumber \\ 
  3 v_9^{\mathrm{h}3} +
  \sum_{i=10}^{n-1}
  \lfloor
  i/3 \rfloor v_i.
\end{align}

If a vertex~$u$ claims halfedges at a vertex~$v$, we say that~$v$
\emph{serves}~$u$. According to~\ref{c:2}, every vertex of degree
three claims three halfedges at a high-degree neighbor. Every degree
four vertex that is adjacent to a hermit together with this hermit
claims four halfedges at a high-degree neighbor by~\ref{c:3}. We sum
up the number of these claims and assess how many of them can be
served by the different types of high-degree vertices.

In general, a high-degree vertex of degree~$i\ge 5$ can serve at
most~$\lfloor i/3\rfloor$ such claims. For~$i\in\{5,6,7,9\}$, we make
 a more detailed analysis, taking into
account the number of adjacent hermits. Specifically, by~\ref{c:3} and~\ref{c:4} a
degree five vertex serves at most one low-degree vertex, which is
either a hermit or a degree-three vertex. A degree six vertex can
serve two degree-three vertices but only if it is not adjacent to a
hermit. If a degree six vertex serves a degree four vertex, it is
adjacent to exactly one hermit by~\ref{c:3}. In particular, a degree
six vertex that is adjacent to two hermits does not serve any degree
three or degree four vertex. Altogether we obtain the following inequality:
\begin{align}\label{eq:deg3}
  v_3 + v_4^{\mathrm{h}1} \leq v_5^{\mathrm{h}0} + 2
  v_6^{\mathrm{h}0} + v_6^{\mathrm{h}1} + 2v_7^{\mathrm{h}0} +
  2v_7^{\mathrm{h}1} + v_7^{\mathrm{h}2} + 2 v_8 + 3 v_9^{\mathrm{h}0} + \nonumber  \\ 
  2 v_9^{\mathrm{h}1} + 2v_9^{\mathrm{h}2} +v_9^{\mathrm{h}3} +
  \sum_{i=10}^{n-1} \lfloor i/3 \rfloor v_i. 
\end{align}
The combination~$(\eqref{eq:deg2}+\eqref{eq:deg3})/2$ together with
\eqref{eq:degsum} yields
\begin{equation}\label{eq:lowbound}
  v_2 + \frac12 v_3 \le \frac12 v_5 + v_6 + \frac32 v_7 + 2v_8 + 2v_9
  + \sum_{i=10}^{n-1} \lfloor i/3 \rfloor v_i.
\end{equation}
Now, using these equations and inequalities, we can prove that
$m - 2n \geq 0$, to complete the proof of \cref{thm:main}. Let us
start from the left hand side, using
\eqref{eq:n}. 
\begin{eqnarray*}
  m - 2n &=& \frac12 \sum_{i=2}^{n-1} i v_i - 2\sum_{i=2}^{n-1} v_i
             \;\;=\;\; \sum_{i=2}^{n-1}\frac{i-4}{2}v_i\\
         &=& -v_2 -\frac12 v_3 + \frac12 v_5 + v_6 + \frac32 v_7 +
             2v_8 + \frac52 v_9 + \sum_{i=10}^{n-1}\frac{i-4}{2}v_i
\end{eqnarray*}
By \eqref{eq:lowbound} the right hand side is nonnegative, quod erat
demonstrandum.\bigskip

\subsection{Admissible Drawings}\label{sec:remarks}

So far we have worked with the abstract graph~$G$. In order to discuss
our charging scheme, we also use a suitably chosen drawing
of~$G$. Specifically, we consider a maximal $2$-planar graph~$G$
on~$n\ge 5$ vertices and a crossing-minimal $2$-plane drawing~$D$
of~$G$ that, among all such drawings, minimizes the number of doubly
crossed edges. We refer to a drawing with these properties as an
\emph{admissible} drawing of~$G$. By \cref{thm:2planar_simple} we know
that~$D$ is simple.

\subsection{Hermits and degree four vertices}\label{sec:hermit}

\begin{restatable}{lemma}{HermitUncrossed}\label{lem:hermit_uncrossed}
  \renewcommand{\mylink}{\statlink{lem:hermit_uncrossed}\prooflink{PHermitUncrossed}}
  Let~$h$ be a hermit and let~$x,y$ be its neighbors in~$G$. Then~$x$
  and~$y$ are adjacent in~$G$ and all three edges~$xy,hx,hy$ are
  uncrossed in~$D$.
\end{restatable}

\noindent
We refer to the edge~$xy$ as the \emph{base} of the hermit~$h$, which
\emph{hosts}~$h$.

\begin{restatable}{lemma}{OneHermitPerEdge}\label{lem:one_hermit_per_edge}
  \renewcommand{\mylink}{\statlink{lem:one_hermit_per_edge}\prooflink{POneHermitPerEdge}}
  Every edge of $G$ hosts at most one hermit.
\end{restatable}

By \cref{lem:efficient_hermit_neighbors} both neighbors of a hermit
have degree at least four. 
A vertex is of type \emph{T4-H} if it has degree exactly four and it
is adjacent to a hermit. The following lemma characterizes these
vertices and ensures that every hermit has at least one high-degree
neighbor.

\begin{restatable}{lemma}{DegreeFourHermit}\label{lem:degree4_hermit}
  \renewcommand{\mylink}{\statlink{lem:degree4_hermit}\prooflink{PDegreeFourHermit}}
  Let~$u$ be a T4-H vertex with neighbors~$h,v,w,x$ in~$G$ such
  that~$h$ is a hermit and~$v$ is the second neighbor of~$h$. Then
  both~$uw$ and~$ux$ are doubly crossed in~$D$, and the two faces
  of~$D\setminus h$ incident to~$uv$ are triangles that are bounded by
  (parts of) edges incident to~$u$ and doubly crossed edges incident
  to~$v$.  Furthermore, we have~$\deg(v)\ge 6$, and if~$\deg(v)=6$,
  then~$h$ is the only hermit adjacent to~$v$ in~$G$.
\end{restatable}

\noindent\begin{minipage}[b]{.82\linewidth}
  \hspace*{1.5em} In our charging scheme, each
  hermit~$h$ claims two halfedges at each high-degree neighbor~$v$:
  the halfedge~$\half{hv}$ and the halfedge~$\half{uv}$, where~$uv$
  denotes the edge that hosts~$h$. Each T4-H vertex~$u$ claims the two
  doubly crossed halfedges at~$v$ that bound the triangular faces
  incident to~$uv$ in~$D$.
\end{minipage}
\begin{minipage}[b]{.18\linewidth}
  \hfill\includegraphics[page=4]{summary_small}
\end{minipage}

\subsection{Degree-three vertices}\label{sec:degreethree}

We distinguish four different types of degree-three vertices in~$G$,
depending on their neighborhood and on the crossings on their incident
edges in~$D$. Consider a degree-three vertex~$u$ in~$G$. By
\cref{lem:vertex_incident_to_uncrossededge} every vertex is incident
to at least one uncrossed edge in~$D$.

\paragraph{\textbf{T3-1: exactly one uncrossed edge.}} The two other edges
incident to~$u$ are crossed.

\begin{restatable}{lemma}{DegreeThreeTwoCrossed}\label{lem:degree3_2crossed_neighbordeg5}
  \renewcommand{\mylink}{\statlink{lem:degree3_2crossed_neighbordeg5}\prooflink{PDegreeThreeTwoCrossed}}
  Let~$u$ be a T3-1 vertex with neighbors~$v,w,x$ in~$G$ such that the
  edge~$uv$ is uncrossed in~$D$. Then 
  the two faces of~$D$ incident to~$uv$ are
  triangles that are bounded by (parts of) edges incident to~$u$ and
  doubly crossed edges incident to~$v$. Furthermore, we
  have~$\deg(v)\ge 5$.
\end{restatable}

\noindent\begin{minipage}[b]{.82\linewidth}
  \hspace*{1.5em} In our charging scheme, each
  T3-1 vertex~$u$ claims three halfedges at its adjacent high-degree
  vertex~$v$: the uncrossed halfedge~$\half{uv}$ along with the two
  neighboring halfedges at~$v$, which are doubly crossed by
  \cref{lem:degree3_2crossed_neighbordeg5}.\par\medskip
\end{minipage}
\begin{minipage}[b]{.18\linewidth}
  \hfill\includegraphics[page=1]{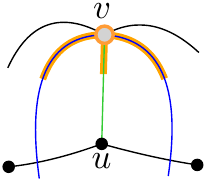}
\end{minipage}

\paragraph{\textbf{T3-2: exactly two uncrossed edges.}}  The third edge
incident to~$u$ is crossed.

\begin{restatable}{lemma}{CRadJ}\label{lem:cradj} 
  \renewcommand{\mylink}{\statlink{lem:cradj}\prooflink{PCRadJ}}
  Let~$u$ be a T3-2 vertex in~$D$ with neighbors~$v,w,x$ such
  that~$uv$ is crossed at a point~$\alpha$. Then~$\alpha$ is the only
  crossing of~$uv$ in~$D$. Further, the edge that crosses~$uv$ is
  doubly crossed, it is incident to~$w$ or~$x$, and its part
  between~$w$ or~$x$ and~$\alpha$ is uncrossed.
\end{restatable}

By the following lemma, we are free to select which of the two
edges~$uv$ and~$ux$ incident to a T3-2 vertex are singly crossed; see
\cref{fig:t3-2stuff}~(left and middle).

\begin{restatable}{lemma}{TTTChoose}\label{prop:t32-choose}
  \renewcommand{\mylink}{\statlink{prop:t32-choose}\prooflink{PTTTChoose}}
  Let~$u$ be a T3-2 vertex in~$D$, and let the neighbors of~$u$ be
  $v,w,x$ such that the edge~$uv$ is (singly) crossed by a (doubly
  crossed) edge~$wb$. Then there exists an admissible drawing~$D'$
  of~$G$ such that (1)~$D'$ is identical to~$D$ except for the
  edge~$wb$ and (2)~the edge~$wb$ crosses the edge~$ux$ in~$D'$.
\end{restatable}

\begin{restatable}{lemma}{DegreeThreeOneCrossed}\label{lem:degree3_1crossed_degrees555}
  \renewcommand{\mylink}{\statlink{lem:degree3_1crossed_degrees555}\prooflink{PDegreeThreeTwoOneCrossed}}
  Let $u$ be a T3-2 vertex with neighbors~$v,w,x$ 
  s.t.~the edge~$uv$ is singly crossed by a doubly crossed edge~$wb$
  in~$D$.
  Then~$\deg(w)\ge 5$ and $\min\{\deg(v),\deg(x)\}\ge 4$.
\end{restatable}

\begin{figure}[htbp]
  \centering\hfill%
  \includegraphics[page=5]{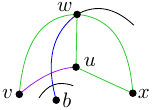}\hfill 
  \includegraphics[page=6]{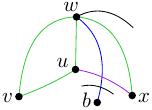}\hfill 
  \includegraphics[page=3]{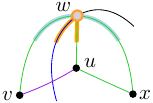}\hfill\hfill  
  \caption{Illustration of \cref{prop:t32-choose} (left and middle);
    halfedges claimed (marked orange) and assessed (marked lightblue)
    by a T3-2 vertex~$u$ (right).}\label{fig:t3-2stuff}
\end{figure}

A halfedge~$\half{wx}$ is \emph{peripheral} for a vertex~$u$ of~$G$ if
(1)~$u$ is a common neighbor of~$w$ and~$x$; (2)~$\deg(w)\ge 5$; and
(3)~$\deg(x)\ge 4$. In our charging scheme, every T3-2 vertex~$u$
claims three halfedges at the adjacent high-degree vertex~$w$: the
halfedge~$\half{uw}$, the doubly crossed halfedge~$\half{bw}$, and one
of the uncrossed peripheral halfedges~$\half{vw}$ or~$\half{xw}$; see
\cref{fig:t3-2stuff}~(right). While the former two are closely tied
to~$u$, the situation is more complicated for the latter two
halfedges.  Eventually, we need to argue that~$u$ can exclusively
claim (at least) one of the two 
peripheral halfedges. But for the time being we say that it
\emph{assesses} both of them.

\paragraph{\textbf{T3-3: all three incident edges uncrossed.}}
We say that such a vertex is of type T3-3. As an immediate consequence
of \cref{lem:uncrossed_edge} each T3-3 vertex~$u$ together with its
neighbors~$\mathrm{N}(u)$ induces a plane~$K_4$ in~$D$. We further
distiguish two subtypes of T3-3 vertices.

The first subtype accounts for the fact that there may be two adjacent
T3-3 vertices in~$D$. We refer to such a pair as an \emph{inefficient
  hermit}.
Observe that two T3-3 vertices~$z,z'$ that form an inefficient hermit
have the same neighbors in~$G\setminus\{z,z'\}$ by
\cref{lem:uncrossed_edge}. A T3-3 vertex that is part of an
inefficient hermit is called a \emph{T3-3 hermit}.

\begin{restatable}{lemma}{DegThreeHermit}\label{lem:deg3_inefficient_hermit}
  \renewcommand{\mylink}{\statlink{lem:deg3_inefficient_hermit}\prooflink{PDegThreeHermit}}
  Let~$z$ be a T3-3 vertex in~$D$, and let~$z'$ be a neighbor of~$z$
  in~$G$ with~$\deg(z')\le 3$. Then~$z'$ is also a T3-3 vertex, that
  is, the pair~$z,z'$ forms an inefficient hermit in~$D$.
\end{restatable}


\begin{restatable}{lemma}{IneffHermitThree}\label{lem:deg3hermit_neighbors}
  \renewcommand{\mylink}{\statlink{lem:deg3hermit_neighbors}\prooflink{PIneffHermitThree}}
  Let $z,z'$ be an inefficient hermit in~$D$, and let~$x,y$ be their
  (common) neighbors in~$G$. Then~$xy$ is an uncrossed edge in~$D$,
  and the degree of~$x$ and~$y$ 
  is at least five each.
\end{restatable}


In particular, \cref{lem:deg3hermit_neighbors} implies that every T3-3
hermit is part of exactly one inefficient hermit. In our charging
scheme, each T3-3 hermit claims three halfedges at one of its (two)
adjacent high-degree vertices. More precisely, let~$z,z'$ be an
inefficient hermit and let~$x,y$ be its neighbors in~$G$. Then the
vertices~$x,y,z,z'$ induce a plane~$K_4$ subdrawing~$Q$ of~$D$. The
vertex~$z$ claims the three halfedges of~$Q$ at~$x$, and~$z'$ claims
the three halfedges of~$Q$ at~$y$.

The second subtype is formed by those T3-3 vertices that are not T3-3
hermits; we call them \emph{T3-3 minglers}.
By \cref{lem:deg3_inefficient_hermit} all neighbors of a T3-3 mingler
have degree at least four.

\begin{restatable}{lemma}{DegreeThreeUncrossed}\label{lem:degree_three_uncrossed_neighbors}
  \renewcommand{\mylink}{\statlink{lem:degree_three_uncrossed_neighbors}\prooflink{PDegreeThreeUncrossed}}
  Let~$u$ be a T3-3 mingler in~$D$,
  and let~$v,w,x$ be its neighbors.
  Then each of~$v,w,x$ has degree at least four. Further, at least one
  vertex among~$v,w,x$ has degree at least six, or at least two
  vertices among~$v,w,x$ have degree at least five.
\end{restatable}

Let~$Q$ denote the plane~$K_4$ induced by~$u,v,w,x$ in~$D$. In our
charging scheme, the T3-3 mingler~$u$ claims the three halfedges
of~$Q$ at one of its high-degree neighbors. That is, the vertex~$u$
assesses all of its (up to six) peripheral halfedges at high-degree
neighbors.

\subsection{The charging scheme}\label{sec:charge}

In this section we argue that our charging scheme works out, that is,
all claims made by low-degree vertices and T4-H vertices can be served
by adjacent high-degree vertices. \cref{fig:summary_low_deg} presents
a summary of the different types of vertices and their claims.

\begin{figure}[htbp]
  \hfill 
  \centering
  \subfigure[T3-1]{
    \includegraphics[page=1]{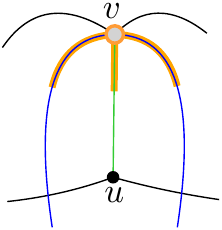}
    \label{fig:summary_low_deg1}}
  \hfill
  \subfigure[T3-2]{
    \includegraphics[page=3]{summary_low_deg}
    \label{fig:summary_low_deg2}
  }
  \hfill
  \subfigure[T3-3]{
    \includegraphics[page=2]{summary_low_deg}
    \label{fig:summary_low_deg3}
  }
  \hfill
  \subfigure[T4-H]{
    \includegraphics[page=4]{summary_low_deg}
    \label{fig:summary_low_deg4}
  }
  \hfill\hfill
  \caption{A vertex~$u$ with~$\deg(u)\in\{3,4\}$ and an adjacent
    high-degree vertex~$v$ at which~$u$ claims halfedges. Claimed
    halfedges are marked orange. Assessed halfedges are marked
    lightblue: A T3-2 vertex claims one of the two lightblue
    peripheral edges, and a T3-3 vertex claims a triple of halfedges
    at one of its high-degree neighbors.}
  \label{fig:summary_low_deg}
\end{figure}

For some halfedges it is easy to see that they are claimed at most
once; these halfedges are shown orange in
\cref{fig:summary_low_deg}. In particular, it is clear that a halfedge
that is incident to the vertex that claims it is claimed at most once.
We also need to consider the claims by hermits, which are not shown in
the figure (except for the hermit adjacent to a T4-H vertex).

\begin{restatable}{lemma}{ChargeHermit}\label{lem:c-hermits}
  \renewcommand{\mylink}{\statlink{lem:c-hermits}\prooflink{PChargeHermit}}
  Every halfedge claimed by a hermit is claimed by this hermit only.
\end{restatable}

The next lemma settles the validity of our charging scheme for T3-1
and T4-H vertices.

\begin{restatable}{lemma}{ClaimDouble}\label{lem:double}
  \renewcommand{\mylink}{\statlink{lem:double}\prooflink{PClaimDouble}}
  Every doubly crossed halfedge is claimed at most once.
\end{restatable}

It remains to argue about the claims to peripheral halfedges by~T3-2
and~T3-3 vertices. Every T3-2 vertex assesses two peripheral halfedges
of which it needs to claim one, and every T3-3 vertex assesses three
pairs of halfedges of which it needs to claim one. In order to find a
suitable assignment of claims for these vertices it is crucial that
not too many vertices compete for the same sets of
halfedges. Fortunately, we can show that this is not the case. We say
that an edge of~$G$ is \emph{assessed} by a low-degree vertex~$u$ if
(at least) one of its corresponding halfedges is assessed by~$u$.

\begin{lemma}\label{lem:claims}
  Every edge is assessed by at most two vertices.
\end{lemma}
\begin{proof}
  For a contradiction consider three vertices~$u_0,u_1,u_2$ of
  type~T3-2 or~T3-3 that assess one of the halfedges of an
  edge~$uv$. Then the edge~$uv$ is uncrossed in~$D$, and all
  of~$u_0,u_1,u_2$ are common neighbors of~$u$ and~$v$
  in~$G$. Moreover, we may suppose that all edges
  between~$u_0,u_1,u_2$ and~$u,v$ are uncrossed in~$D$: For T3-3
  vertices all incident edges are uncrossed, anyway, and for T3-2
  vertices this follows by \cref{prop:t32-choose}. In other words, we
  have a plane~$K_{2,3}$ subdrawing~$B$ in~$D$ between~$u_0,u_1,u_2$
  and~$u,v$. Let~$\phi_0,\phi_1,\phi_2$ denote the three faces of~$B$
  such that~$\partial\phi_i=uu_ivu_{i\oplus 1}$, where~$\oplus$
  denotes addition modulo~$3$.

  Consider some~$i\in\{0,1,2\}$. As~$\deg(u_i)=3$, there is exactly
  one vertex~$x_i\notin\{u,v\}$ that is adjacent to~$u_i$ in~$G$. The
  edge from~$u_i$ to~$x_i$ in~$D$ enters the interior of exactly one
  of~$\phi_i$ or~$\phi_{i\oplus 2}$. In other words, 
  for exactly one of~$\phi_i$ or~$\phi_{i\oplus 2}$,
  no edge incident to~$u_i$ enters its interior.
  It follows that~$G^-:=G\setminus\{u,v\}$ is
  disconnected, in particular, the vertices~$u_0,u_1,u_2$ are split
  into at least two components. Suppose without loss of generality
  that~$u_0$ is separated from both~$u_1$ and~$u_2$ in~$G^-$, and
  let~$C_0$ denote the component of~$G^-$ that contains~$u_0$.
  Let~$D_0$ denote the subdrawing of~$D$ induced by~$C_0$ along with
  all edges between~$C_0$ and~$u,v$. Observe that~$uu_0v$ is an
  uncrossed path along the outer face of~$D_0$.

  We remove~$D_0$ from~$D$ and put it back right next to the uncrossed
  path~$uu_1v$, in the face ($\phi_0$ or~$\phi_1$) incident to~$u_1$
  that is not entered by any edge incident to~$u_1$; see
  \cref{fig:claims} for illustration. Furthermore, we flip~$D_0$ with
  respect to~$u,v$ if necessary so as to ensure that the two uncrossed
  paths~$uu_1v$ and~$uu_0v$ appear consecutively in the circular order
  of edges incident to~$u$ and~$v$, respectively, in the resulting
  drawing~$D'$, effectively creating a quadrilateral
  face~$uu_1vu_0$. The drawing~$D'$ is an admissible drawing of~$G$,
  to which we can add an uncrossed edge~$u_0u_1$ in the
  face~$uu_1vu_0$, a contradiction to the maximality
  of~$G$. Therefore, no such triple~$u_0,u_1,u_2$ of vertices exists
  in~$G$.
\end{proof}

\begin{figure}[htbp]
  \hfill 
   \centering
    \subfigure[]{
    \includegraphics[page=1]{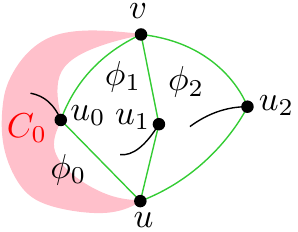}
 }
  \hfill
  \subfigure[]{
    \includegraphics[page=2]{claims}
  }
  \hfill\hfill
  \caption{Redrawing 
    in case that three vertices~$u_0,u_1,u_2$ claim a halfedge of the
    edge~$uv$.}
  \label{fig:claims}
\end{figure}

Note that \cref{lem:claims} settles the claims by T3-3 hermits, as
they come in pairs that assess the same halfedges. By
\cref{lem:claims} no other vertex assesses these halfedges, so our
scheme of assigning the halfedges at one endpoint to each works
out. It remains to consider T3-2 vertices and T3-3 minglers. Let us
start with the T3-2 vertices.
Consider an edge or halfedge~$e$ that is assessed by a low-degree
vertex~$u$. We say that~$e$ is \emph{contested} if there exists
another low-degree vertex~$u'\ne u$ that also assesses~$e$. An edge or
halfedge that is not contested is \emph{uncontested}.

\begin{lemma}\label{lem:t3-2claims}
  The claims of all T3-2 vertices can be resolved in a greedy manner.
\end{lemma}
\begin{proof}
  Let~$u$ be a T3-2 vertex in~$D$, and let~$\half{wv}$ and~$\half{xv}$
  denote the halfedges that~$u$ assesses.
  We start a sequence of greedy selections for the claims of
  vertices~$u_1,u_2,\ldots,u_k$ by letting~$u_1:=u$ claim one
  of~$\half{wv}$ and~$\half{xv}$ arbitrarily, say, let~$u$
  claim~$\half{wv}$ (and withdraw its assessment of~$\half{xv}$).
  More generally, at the~$i$-th step of our selection procedure we
  have a vertex~$u_i$ that has just claimed one of its assessed
  halfedges~$\half{w_iv_i}$. By \cref{lem:claims} there is at most one
  other vertex~$u_{i+1}$ that also assesses~$\half{w_iv_i}$. If no
  such vertex~$u_{i+1}$ exists, then we are done and the selection
  procedure ends here with~$i=k$.  Otherwise, we consider two cases.

  \begin{case}[$u_{i+1}$ is a T3-2 vertex]
    Then there is only one other (than~$\half{w_iv_i}$) halfedge
    that~$u_{i+1}$ assesses, denote it by~$\half{x_{i+1}v_i}$. We
    let~$u_{i+1}$ claim~$\half{x_{i+1}v_i}$ and proceed with the next
    step.
  \end{case}

  \begin{case}[$u_{i+1}$ is a T3-3 mingler]
    Then~$u_{i+1}$ also assesses~$\half{v_iw_i}$, which is uncontested
    now, and it also assesses a second halfedge~$\half{x_{i+1}w_i}$
    at~$w_i$. We let~$u_{i+1}$ claim both~$\half{v_iw_i}$
    and~$\half{x_{i+1}w_i}$ and then proceed with the next step.
  \end{case}

  For the correctness of the selection procedure it suffices to note
  that at every step exactly one halfedge is claimed that is (still)
  contested, and the claims of the (unique) vertex that assesses this
  halfedge are resolved in the next step. 
  In particular, at the end of the 
  procedure, all (still) assessed edges are unclaimed. As long as
  there exists another T3-2 vertex in~$D$ that has not claimed one of
  the two halfedges it requires, we start another selection procedure
  from there. Thus, eventually the claims of all T3-2 vertices are
  resolved.
\end{proof}

At this point it only remains to handle the claims of the remaining T3-3
minglers. They are more tricky to deal with compared to the T3-2
vertices because they require \emph{two} halfedges at a single
high-degree vertex. We may restrict our attention to a
subclass of T3-3 minglers which we call tricky, as they assess a
directed $3$-cycle of contested halfedges. Consider a T3-3
mingler~$u$, and let~$v,w,x$ be its neighbors in~$G$.  We say that~$u$
is \emph{tricky} if (1)~it assesses all six halfedges among its
neighbors and (2)~all of the halfedges~$\half{vw},\half{wx},\half{xv}$
or all of the halfedges~$\half{vx},\half{xw},\half{wv}$ (or both sets)
are contested. A T3-3 mingler that is not tricky is \emph{easy}.


\begin{lemma}\label{lem:t3-3-easy}
  The claims of all easy T3-3 minglers can be resolved in a
  postprocessing step.
\end{lemma}
\begin{proof}
  Let~$M$ denote the set of easy T3-3 minglers in~$D$. We remove~$M$
  along with all the corresponding assessments from consideration, and
  let all other (that is, tricky) T3-3 minglers make their claims. We
  make no assumption about preceding claims, other than that every
  vertex (1)~claims edges incident to one vertex only and (2)~claims
  only edges it assesses. After all tricky T3-3 minglers have made
  their claims, we process the vertices from~$M$, one by one, in an
  arbitrary order. In the following, the terms \emph{assessed} and
  \emph{(un)contested} refer to the initial situation, before any
  claims were made. The current state of a halfedge is described as
  either \emph{claimed} or \emph{unclaimed}. 
  
  Consider a vertex $u\in M$. If not all six halfedges are assessed by
  $u$, then not all of its neighbors are high-degree,
  in which case, at most one peripheral edge of $u$ is contested.
  Thus, there always exists one pair of halfedges that is unclaimed and
  can be claimed by $u$.
  In the other case, let~$H$ denote the set of six
  halfedges that are assessed by~$u$. By~(2) every uncontested
  halfedge in~$H$ is unclaimed. By \cref{lem:claims} every edge is
  assessed by at most two vertices. Thus, for each of the
  edges~$vx,xw,wv$ at most one vertex other than~$u$ assesses this
  edge. This other vertex may have claimed a corresponding halfedge,
  but by~(1) for every edge~$vx,xw,wv$ at least one of its two
  corresponding halfedges is unclaimed.

  As~$u$ is easy, at least one of~$\half{vw},\half{wx},\half{xv}$ and
  at least one of~$\half{vx},\half{xw},\half{wv}$ is
  uncontested. Suppose without loss of generality that~$\half{vw}$ is
  uncontested. We conclude with three cases.

  \begin{case}[$\half{vx}$ is uncontested] 
    At least one of~$\half{xw}$ or~$\half{wx}$ is unclaimed. Thus, we
    can let~$u$ claim one of the pairs~$\half{xw},\half{vw}$
    or~$\half{wx},\half{vx}$.
  \end{case}
  
  \begin{case}[$\half{xw}$ is uncontested]
    Then we let~$u$ claim~$\half{vw},\half{xw}$.
  \end{case}  

  \begin{case}[$\half{wv}$ is uncontested]
    If one of~$\half{xw}$ or~$\half{xv}$ is unclaimed, then we let~$u$
    claim it together with the matching halfedge of the edge~$vw$,
    which is uncontested by assumption. Otherwise, both~$\half{xw}$
    and~$\half{xv}$ are claimed. Then both~$\half{wx}$ and~$\half{vx}$
    are unclaimed, and so~$u$ can safely claim these two
    halfedges.  
  \end{case}
  %
\end{proof}

It remains to resolve the claims of tricky T3-3 minglers. Note that
the classification tricky vs.~easy depends on the other T3-3
minglers. For instance, a T3-3 mingler that is tricky initially may
become easy after removing another easy T3-3 mingler. Here we have to
deal with those T3-3 minglers only that remain tricky after all easy
T3-3 minglers have been iteratively removed from consideration.

\begin{lemma}\label{lem:t3-3-tricky}
  The claims of all tricky T3-3 minglers can be resolved in a greedy manner.
\end{lemma}
\begin{proof}
  Let~$u$ be a tricky T3-3 mingler, and let~$v,w,x$ be its neighbors
  in~$G$. As for each tricky T3-3 mingler all three peripheral edges
  are contested by other tricky T3-3 minglers, there exists a circular
  sequence~$u=u_1,\ldots,u_k$, with~$k\ge 2$, of tricky T3-3 minglers
  that are neighbors of~$v$ in~$G$ and whose connecting edges appear
  in this order around~$v$ in~$D$. We distinguish two cases, depending
  on the parity of~$k$.

  \begin{case}[$k$ is even]\label{case:1}
    Then we let each~$u_i$, for~$i$ odd, claim the two halfedges
    at~$v$ that it assesses. This resolves the claims for all~$u_i$,
    with~$i$ odd, and we claim that now all~$u_i$, with~$i$ even, are
    easy. To see this, consider a vertex~$u_i$, with~$i$ even. Both of
    its assessed halfedges at~$v$ are now claimed; denote these
    halfedges by~$\half{w_iv}$ and~$\half{x_iv}$. It follows that
    both~$\half{vw_i}$ and~$\half{vx_i}$ are unclaimed and~$u_i$ is
    the only vertex that still assesses them. As there is no directed
    $3$-cycle of contested halfedges among the halfedges assessed
    by~$u_i$ anymore, the vertex~$u_i$ is easy.
  \end{case}

  \begin{case}[$k$ is odd]
    Then we let each~$u_i$, for~$i<k$ odd, claim the two halfedges
    at~$v$ that it assesses. This resolves the claims for these~$u_i$
    and makes all~$u_i$, with~$i<k-1$ even, easy, as in \cref{case:1}
    above. It remains to argue about~$u_{k-1}$ and~$u_k$. Let~$x_iv$
    denote the edge assessed by both~$u_i$ and~$u_{i+1}$,
    for~$1\le i<k$, and let~$x_kv$ denote the edge assessed by
    both~$u_k$ and~$u_1$. As~$\half{x_kv}$ is claimed by~$u_1$, we can
    let~$u_k$ claim~$\half{vx_k}$ and~$\half{x_{k-1}x_k}$ along with
    it. This makes~$u_{k-1}$ easy, as both~$\half{vx_{k-2}}$
    and~$\half{vx_{k-1}}$ are uncontested now. However, we still need
    to sort out the bold claim on~$\half{x_{k-1}x_k}$ by~$u_k$. To
    this end, we apply the same greedy selection procedure as in the
    proof of \cref{lem:t3-2claims}, except that here we start with the
    selection of~$\half{x_{k-1}x_k}$, as the only contested halfedge
    that is claimed, and here we can only encounter (tricky) T3-3
    minglers over the course of the procedure.
  \end{case}

  We get rid of at least two tricky T3-3 minglers, either by resolving
  their claims or by making them easy. Thus, after a finite number of
  steps, no tricky T3-3 mingler remains.
\end{proof}

Our analysis of the charging scheme is almost complete now. However,
we still need to justify our claim about degree five vertices in
Property~\ref{c:4}. In principle it would be possible that two
halfedges at a degree five vertex are claimed by a hermit and the
remaining three by a degree-three vertex. But we can show that this is
impossible.

\begin{restatable}{lemma}{DegreeFive}\label{lem:deg5_atmost_one_low_deg}
  \renewcommand{\mylink}{\statlink{lem:deg5_atmost_one_low_deg}\prooflink{PDegreeFive}}
  At most one low-degree vertex claims halfedges at a degree five
  vertex.
\end{restatable}

\section{The Upper Bound: Proof outline of \cref{thm:upper}} \label{sec:upper_bound_sketch}

In this section we describe a construction for a family of maximal
$2$-planar graphs with few edges. We give a complete description of
this family. 
But due to space constraints we give a very rough sketch only for the
challenging part of the proof: to show that these graphs are
\emph{maximal} $2$-planar.  \cref{sec:upper_bound_full} provides a
complete version of this section, with all proofs.

The graphs can roughly be described as braided cylindrical grids. More
precisely, for a given~$k\in\N$ we construct our graph~$G_k$
on~$10k+140$ vertices as follows.
\begin{itemize}
\item Take~$k$ copies of~$C_{10}$, the cycle on~$10$ vertices, and
  denote them by~$D_1,\ldots,D_k$. Denote the vertices of~$D_i$,
  for~$i\in\{1,\ldots,10\}$, by~$v_0^i,\ldots,v_9^i$ so that the edges
  of~$D_i$ are~$\{v_j^iv_{j\oplus 1}^i\colon 0\le j\le 9\}$,
  where~$\oplus$ denotes addition modulo~$10$.
\item For every~$i\in\{1,\ldots,k-1\}$, connect the vertices of~$D_i$
  and~$D_{i+1}$ by a braided matching, as follows. For~$j$ even, add
  the edge~$v_j^iv_{j\oplus 8}^{i+1}$ to~$G_k$ and for~$j$ odd, add
  the edge~$v_j^iv_{j\oplus 2}^{i+1}$ to~$G_k$. See
  \cref{fig:nested_decagons:1}~(left) for illustration.
\item To each edge of~$D_1$ and~$D_k$ we attach a
  gadget~$X\simeq K_9\setminus(K_2+K_2+P_3)$ so as to forbid crossings
  along these edges. Denote the vertices of~$X$ by~$x_0,\ldots,x_8$
  such that~$\deg_X(x_0)=\deg_X(x_1)=8$, $\deg_X(x_8)=6$ and all other
  vertices have degree seven. Let~$x_6,x_7$ be the non-neighbors
  of~$x_8$. To an edge~$e$ of~$D_1$ and~$D_k$ we attach a copy of~$X$ so
  that~$e$ takes the role of the edge~$x_6x_7$ in this copy of~$X$. As
  altogether there are~$20$ edges in~$D_1$ and~$D_k$ and each copy
  of~$X$ adds seven more vertices, a total of~$20\cdot 7=140$ vertices
  are added to~$G_k$ with these gadgets.
\item Finally, we add the edges~$v_j^iv_{j\oplus 2}^i$, for
  all~$0\le j\le 9$ and~$i\in\{1,k\}$.
\end{itemize}
This completes the description of the graph~$G_k$. Note that~$G_k$
has~$10k+140$ vertices and~$10k+10(k-1)+20\cdot 31+2\cdot 10=20k+630$
edges. So to prove \cref{thm:upper} asymptotically it suffices to
choose~$c\ge 630-2\cdot 140=350$ and show that~$G_k$ is maximal
$2$-planar. Using some small local modifications we can then obtain
the statement for all values of~$n$.

To show that~$G_k$ is $2$-planar it suffices to give a $2$-plane
drawing of it. Such a drawing can be deduced from
\cref{fig:nested_decagons:1}: (1)~We nest the cycles~$D_1,\ldots,D_k$
with their connecting edges using the drawing depicted in
\cref{fig:nested_decagons:1}~(left), (2)~draw all copies of~$X$
attached to the edges of~$D_1$ and~$D_k$ using the drawing depicted in
\cref{fig:nested_decagons:1}~(right), and (3)~draw the remaining edges
among the vertices of~$D_1$ and~$D_k$ inside and outside~$D_1$
and~$D_k$, respectively.

\begin{figure}[htbp]
  \hfill
  \begin{minipage}[t]{.48\linewidth}
    \includegraphics[page=1]{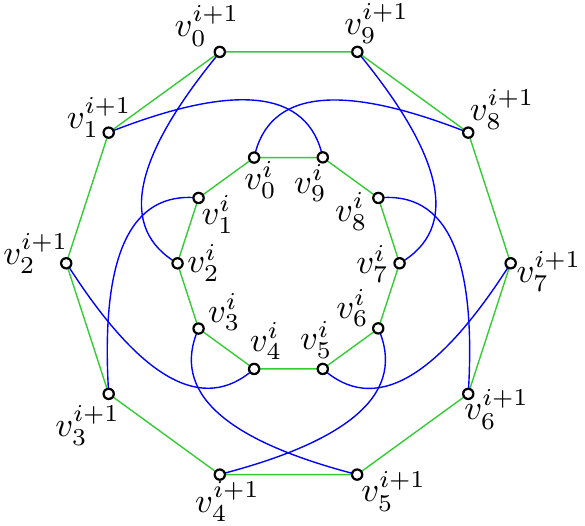}
  \end{minipage}
  \hfill
  \begin{minipage}[t]{.48\linewidth}
    \includegraphics[page=2]{decagons}
  \end{minipage}
  \hfill
  \caption{The braided matching between two consecutive ten-cycles in~$G_k$, shown in blue (left); the gadget graph~$X$ that we attach to the edges of the first and the last ten-cycle of~$G_k$ (right).\label{fig:nested_decagons:1}}
\end{figure}

It is much more challenging, though, to argue that~$G_k$ is
\emph{maximal} $2$-planar. In fact, we do not know of a direct
argument to establish this claim. Instead, we prove that~$G_k$ admits
essentially only one $2$-plane drawing, which is the one described
above. Then maximality follows by just inspecting this drawing and
observing that no edge can be added there because every pair of
non-adjacent vertices is separated by a cycle of doubly-crossed edges.

This leaves us having to prove that~$G_k$ has a unique $2$-plane
drawing.
We solve the problem in a somewhat brute-force way: by enumerating all
$2$-plane drawings of~$G_k$, using computer support.  Still, it is not
immediately clear how to do this, given that (1)~the space of (even
$2$-plane) drawings of a graph can be vast; (2)~$(G_k)$ is an infinite
family; and (3)~already for small~$k$, even a single graph~$G_k$ is
quite large.

First of all, the gadget graph~$X$ is of small constant size. So it
can be analyzed separately, and there is no need to explicitly include
these gadgets into the analysis of~$G_k$. Instead, we account for the
effect of the gadgets 
by considering the edges of~$D_1$ and~$D_k$ as uncrossable. In this
way, we also avoid counting all the variations of placing the attached
copy of~$X$ on either side of the corresponding cycle as different
drawings. In fact, the gadget~$X$ itself has a few formally different
$2$-plane drawings due to its automorphisms. But for our purposes of
arguing about the maximality of~$G_k$, these differences do not
matter. However, these variations are the reason that the $2$-plane
drawing of~$G_k$ is \emph{essentially} unique only.

We also disregard the length two edges along~$D_1$ and~$D_k$. Denote
the resulting subgraph of~$G_i$ by~$G_i^-$. We iteratively enumerate
the $2$-plane drawings of~$G_i^-$, for all~$i\in\N$, where only the
edges of the first cycle~$D_1$ are labeled uncrossable (but not the
edges of the last cycle~$D_i$). All drawings are represented as a
doubly-connected edge list (DCEL)~\cite[Chapter~2]{bkos-cgaa-08}. As a
base case, we use the unique (up to orientation, which we select to be
counterclockwise, without loss of generality) plane drawing
of~$G_1^-=D_1$. For each drawing~$\Gamma$ computed, with a specific
ten-cycle of vertices labeled as~$D$, we consider all possible ways to
extend~$\Gamma$ by adding another ten-cycle of new, labeled vertices
and connect it to~$D$ using a braided matching, as in the construction
of~$G_k$ and depicted in \cref{fig:nested_decagons:1}~(left).

So in each iteration we have a partial drawing~$\Gamma$ and a
collection~$H$ of vertices and edges still to be drawn. We then
exhaustively explore the space of simple $2$-plane drawings
of~$\Gamma\cup H$. Our approach is similar to the one used by
Angelini, Bekos, Kaufmann, and Schneck~\cite{AngeliniBKS20} for
complete and complete bipartite graphs. We consider the edges to be
drawn in some order such that whenever an edge is considered, at least
one of its endpoints is in the drawing already. When drawing an edge,
we go over (1)~all possible positions in the rotation at the source
vertex and for each such position all options to (2)~draw the edge
with zero, one or two crossings. Each option to consider amounts to a
traversal of some face incident to the source vertex, and up to two
more faces in the neighborhood.  At every step we ensure that the
drawing constructed remains $2$-plane and simple, and backtrack
whenever an edge cannot be added or the drawing is complete (that is,
it is a $2$-plane drawing of~$G_i^-$, for some~$i\in\N$).

Every drawing for~$\Gamma\cup H$ obtained in this fashion is then
tested, as described below. If the tests are successful, then the
drawing is added to the list of drawings to be processed, as a child
of~$\Gamma$, and such that the ten-cycle in~$H$ takes the role of~$D$
for future processing.

As for the testing a drawing~$\Gamma$, we are only interested in a
drawing that can---eventually, after possibly many iterations---be
extended in the same way with an uncrossable ten-cycle~$D_k$. In
particular, all vertices and edges of~$D_k$ must lie in the same face
of~$\Gamma$. Hence, we test whether there exists a suitable
\emph{potential final} face in~$\Gamma$ where~$D_k$ can be placed (see
\cref{sec:upper_bound_full} for details); if not, then we
discard~$\Gamma$ from further consideration. We also go over the faces
of~$\Gamma$ and remove \emph{irrelevant} faces and vertices that are
too far from any potential final face to ever be able to interact with
vertices and edges to be added in future iterations. Finally, we check
whether the resulting \emph{reduced} drawing has already been
discovered by comparing it to all the already discovered drawings (by
testing for an isomorphism that preserves the cycle~$D$). If not, then
we add it to the list of valid drawings.

For each drawing for which we found at least one child drawing, we
also test whether there exists a similar extension where the cycle
in~$H$ is uncrossable. Whenever such an extension is possible, we
found a $2$-plane drawing of~$G_i$, for some~$i\in\N$. The algorithm
for~$G_i^-$ runs for about~$1.5$ days and discovers~$86$ simple
$2$-plane drawings of~$G_i^-$. In only one of these drawings the last
ten-cycle is uncrossed: the drawing described above (see
\cref{fig:nested_decagons:1}). The algorithm for the gadget~$X$ runs
for about $3$min.~and discovers $32$ simple $2$-plane drawings, as
expected. The full source code is available in our
repository~\cite{our-repo}.

\section{Conclusions}\label{sec:conclusions}

We have obtained tight bounds on the number of edges in maximal
$2$-planar graphs, up to an additive constant. Naturally, one would
expect that our approach can also be applied to other families of
near-planar graphs, specifically, to 
maximal~$1$- and $3$-planar graphs. Intuitively, 
for $k$-planar graphs the challenge with increasing~$k$ is that the
structure of the drawings gets more involved, whereas with
decreasing~$k$ we aim for a higher bound.

\bibliographystyle{splncs04}
\bibliography{arxiv.bib}

\begin{thebibliography}{10}
\providecommand{\url}[1]{\texttt{#1}}
\providecommand{\urlprefix}{URL }
\providecommand{\doi}[1]{https://doi.org/#1}

\bibitem{ackerman2019topological}
Ackerman, E.: On topological graphs with at most four crossings per edge.
  Computational Geometry  \textbf{85},  101574 (2019).
  \doi{10.1016/j.comgeo.2019.101574}

\bibitem{AngeliniBKS20}
Angelini, P., Bekos, M.A., Kaufmann, M., Schneck, T.: Efficient generation of
  different topological representations of graphs beyond-planarity. J. Graph
  Algorithms Appl.  \textbf{24}(4),  573--601 (2020). \doi{10.7155/jgaa.00531},
  \url{https://doi.org/10.7155/jgaa.00531}

\bibitem{brandenburg_2planar}
Auer, C., Brandenburg, F., Glei{\ss}ner, A., Hanauer, K.: On sparse maximal
  2-planar graphs. In: Proc. 20th Int. Sympos. Graph Drawing ({GD} 2012).
  Lecture Notes Comput. Sci., vol.~7704, pp. 555--556. Springer (2012).
  \doi{10.1007/978-3-642-36763-2\_50},
  \url{https://doi.org/10.1007/978-3-642-36763-2\_50}

\bibitem{BaratT18}
Bar{\'{a}}t, J., T{\'{o}}th, G.: Improvements on the density of maximal
  1-planar graphs. J. Graph Theory  \textbf{88}(1),  101--109 (2018).
  \doi{10.1002/jgt.22187}

\bibitem{bt-s2dfe-21}
Bar{\'{a}}t, J., T{\'{o}}th, G.: Saturated $2$-planar drawings with few edges.
  CoRR  \textbf{abs/2110.12781} (2021), \url{https://arxiv.org/abs/2110.12781}

\bibitem{bekos_optimal2planargraphs}
Bekos, M.A., Kaufmann, M., Raftopoulou, C.N.: On optimal 2- and 3-planar
  graphs. In: Proc. 33rd Internat. Sympos. Comput. Geom. ({SoCG 2017}). LIPIcs,
  vol.~77, pp. 16:1--16:16. Schloss Dagstuhl - Leibniz-Zentrum f{\"{u}}r
  Informatik (2017). \doi{10.4230/LIPIcs.SoCG.2017.16},
  \url{https://doi.org/10.4230/LIPIcs.SoCG.2017.16}

\bibitem{bkos-cgaa-08}
de~Berg, M., van Kreveld, M., Overmars, M., Schwarzkopf, O.: Computational
  Geometry: Algorithms and Applications. Springer, Berlin, Germany, 3rd edn.
  (2008). \doi{10.1007/978-3-540-77974-2}

\bibitem{bsw-bsgr-83}
Bodendiek, R., Schumacher, H., Wagner, K.: Bemerkungen zu einem
  {S}echsfarbenproblem von {G.}~{R}ingel. Abhandlungen aus dem Mathemati\-schen
  Seminar der Universit\"at Hamburg  \textbf{53},  41--52 (1983).
  \doi{10.1007/BF02941309}

\bibitem{brandenburg_1planar}
Brandenburg, F., Eppstein, D., Glei{\ss}ner, A., Goodrich, M.T., Hanauer, K.,
  Reislhuber, J.: On the density of maximal 1-planar graphs. In: Proc. 20th
  Int. Sympos. Graph Drawing ({GD} 2012). Lecture Notes Comput. Sci.,
  vol.~7704, pp. 327--338. Springer (2012).
  \doi{10.1007/978-3-642-36763-2\_29},
  \url{https://doi.org/10.1007/978-3-642-36763-2\_29}

\bibitem{ckpru-eskd-21}
Chaplick, S., Klute, F., Parada, I., Rollin, J., Ueckerdt, T.: Edge-minimum
  saturated k-planar drawings. In: Proc. 29th Int. Sympos. Graph Drawing ({GD}
  2021). Lecture Notes Comput. Sci., vol. 12868, pp. 3--17. Springer (2021).
  \doi{10.1007/978-3-030-92931-2\_1},
  \url{https://doi.org/10.1007/978-3-030-92931-2\_1}

\bibitem{dlm-sgdbp-20}
Didimo, W., Liotta, G., Montecchiani, F.: A survey on graph drawing beyond
  planarity. {ACM} Comput. Surv.  \textbf{52}(1),  1--37 (2020).
  \doi{10.1145/3301281}

\bibitem{hirs-sstgfe-18}
Hajnal, P., Igamberdiev, A., Rote, G., Schulz, A.: Saturated simple and
  2-simple topological graphs with few edges. J. Graph Algorithms Appl.
  \textbf{22}(1),  117--138 (2018). \doi{10.7155/jgaa.00460},
  \url{https://doi.org/10.7155/jgaa.00460}

\bibitem{our-repo}
Hoffmann, M., Reddy, M.M.: The number of edges in maximal 2-planar
  graphs---code repository.
  \url{https://github.com/meghanamreddy/Sparse-maximal-2-planar-graphs}

\bibitem{ht-bpg-20}
Hong, S.H., Tokuyama, T. (eds.): Beyond Planar Graphs. Springer, Singapore
  (2020). \doi{10.1007/978-981-15-6533-5}

\bibitem{hms-pm1pg-12}
Hud{\'a}k, D., Madaras, T., Suzuki, Y.: On properties of maximal 1-planar
  graphs. Discussiones Mathematicae  \textbf{32},  737--747 (2012).
  \doi{10.7151/dmgt.1639}

\bibitem{k-mn1pg-08}
Korzhik, V.P.: Minimal non-1-planar graphs. Discrete Math.  \textbf{308}(7),
  1319--1327 (2008). \doi{10.1016/j.disc.2007.04.009}

\bibitem{KYNCL2015295}
Kyn\v{c}l, J., Pach, J., Radoi\v{s}i\'{c}, R., T\'{o}th, G.: Saturated simple
  and k-simple topological graphs. Computational Geometry  \textbf{48}(4),  295
  -- 310 (2015). \doi{10.1016/j.comgeo.2014.10.008}

\bibitem{PachRTT06}
Pach, J., Radoi{\v{c}}i{\'{c}}, R., Tardos, G., T{\'{o}}th, G.: Improving the
  crossing lemma by finding more crossings in sparse graphs. Discrete Comput.
  Geom.  \textbf{36}(4),  527--552 (2006). \doi{10.1007/s00454-006-1264-9}

\bibitem{PachT97}
Pach, J., T{\'{o}}th, G.: Graphs drawn with few crossings per edge.
  Combinatorica  \textbf{17}(3),  427--439 (1997). \doi{10.1007/BF01215922},
  \url{https://doi.org/10.1007/BF01215922}

\bibitem{Ringel65}
Ringel, G.: Ein {S}echsfarbenproblem auf der {K}ugel. Abhandlungen aus dem
  Mathematischen Seminar der Universit\"at Hamburg  \textbf{29},  107--117
  (1965). \doi{10.1007/BF02996313}

\bibitem{schaefer2012graph}
Schaefer, M.: The graph crossing number and its variants: {A} survey. The
  Electronic Journal of Combinatorics  \textbf{20} (2013). \doi{10.37236/2713},
  version~7 (April~8, 2022)

\bibitem{s-o1pgtos-10}
Suzuki, Y.: Optimal 1-planar graphs which triangulate other surfaces. Discr.
  Math.  \textbf{310}(1),  6--11 (2010). \doi{10.1016/j.disc.2009.07.016}

\end{thebibliography}

\appendix

\section{Basic properties of maximal $k$-planar graphs}\label{sec:basic}

\begin{lemma}\label{lem:connected}
  Every maximal $k$-planar graph is connected.
\end{lemma}
\begin{proof}\label{PConnected}
  Suppose for a contradiction that a maximal $k$-planar graph~$G$ has two components~$A,B$. Consider (separate)~$k$-plane drawings~$\gamma_A$ of~$A$ and~$\gamma_B$ of~$B$, and pick a pair of vertices~$a\in A$ and~$b\in B$. Let~$f_A$ be a face of~$\gamma_A$ incident to~$a$, and let~$f_B$ be a face of~$\gamma_B$ incident to~$b$. If necessary adjust the drawing~$\gamma_B$ so that~$f_B$ is its outer face, and then scale~$\gamma_B$ appropriately to place it inside~$f_A$. To the resulting drawing of~$G[A\cup B]$ we may add an uncrossed edge~$ab$, in contradiction to the maximality of~$G$.
\end{proof}

\UncrossedEdge*
\begin{proof}\label{PUncrossedEdge}
  An uncrossed edge~$uv$ can be added to~$D$ inside the common face. So if~$uv$ is not an edge of~$G$, this is a contradiction to the maximality of~$G$. And if~$uv$ is crossed in~$D$, we replace the original drawing with the uncrossed one. The resulting drawing~$D'$ has strictly fewer crossings than~$D$ overall, and also locally every edge has at most as many crossings in~$D'$ as it has in~$D$. In particular, the drawing~$D'$ is~$k$-plane, which is a contradiction to the crossing-minimality of~$D$.
\end{proof}

\VertexIncidentToUncrossedEdge*
\begin{proof}\label{PVertexIncidentToUncrossedEdge}
  Let $u$ be a vertex of~$D$. By \cref{lem:connected} there is at least one edge~$e=uv$ incident to~$u$ in~$D$. If~$e$ is uncrossed in~$D$, then the statement of the lemma holds.

  Hence assume that~$e$ is crossed in~$D$. Let $f=xy$ be the first edge that crosses~$e$ when traversing it starting from $u$, and denote this crossing by~$\alpha$; see \cref{fig:2connected_fig1} for illustration. By \cref{thm:2planar_simple} we know that~$D$ is simple and, therefore, the vertices~$x,y,u,v$ are pairwise distinct.
  Edge $f$ has at most two crossings, and one of its crossings is with~$e$. 
  Without loss of generality assume that the other crossing on~$f$, if there is any, lies between~$x$ and~$\alpha$. 
  Then, the edge~$f$ is uncrossed between~$\alpha$ and~$y$. 
  In addition, by definition of~$f$ the edge~$e$ is uncrossed between~$u$ and~$\alpha$. 
  Thus, the vertices~$u$ and~$y$ lie on a common face in~$D$. Then by \cref{lem:uncrossed_edge} there is an uncrossed edge~$uy$ in~$D$, as claimed.
\end{proof}

\begin{figure}[htbp]
    \centering
    \includegraphics[scale=1]
    {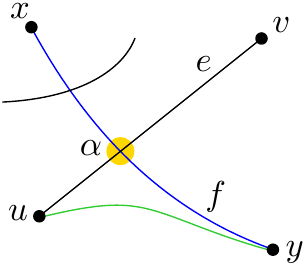}
    \caption{Proof of \cref{lem:vertex_incident_to_uncrossededge}: Vertices $u$ and $y$ lie on a common face.}
  \label{fig:2connected_fig1}
\end{figure}

\TwoConnected*
\begin{proof}\label{PTwoConnected}
  Suppose for a contradiction that~$G=(V,E)$ is a maximal $k$-planar graph that has a cut vertex~$v$. 
  Removing $v$ from $G$ divides the graph into two or more connected components. Denote one of these components by~$A$, and let~$B=V\setminus A$.
  Then both~$H_0=G[A\cup\{v\}]$ and~$H_1=G[B]$ are~$k$-planar (as subgraphs of the~$k$-planar graph~$G$).
  Let~$D_i$ be a crossing-minimal~$k$-plane drawing of~$H_i$, for~$i\in\{0,1\}$. Then by \cref{lem:vertex_incident_to_uncrossededge} there is an uncrossed edge~$vu_i$ in~$D_i$, for all~$i\in\{0,1\}$. As~$H_0$ and~$H_1$ are edge-disjoint, we have~$u_0\ne u_1$.

  We create a~$k$-plane drawing of~$G$ as follows. Place~$v$ arbitrarily and assign to each~$H_i$, for~$i\in\{0,1\}$, a private open halfplane incident to~$v$. If necessary adjust~$D_i$ so that~$vu_i$ lies on its outer face, and continuously deform~$D_i$ so that~$v$ is an extreme point (that is, so that there is a line~$\ell$ through~$v$ such that~$D_i\setminus v$ lies strictly on one side of~$\ell$). Then place a copy of~$D_i$ inside its assigned halfplane so that it attaches to~$v$, for all~$i\in\{0,1\}$. The result is a~$2$-plane drawing~$D'$ of~$G$, where both~$vu_0$ and~$vu_1$ are edges of the outer face. Thus, we can add an uncrossed edge~$u_0u_1$ to~$D'$, which is a contradiction to the maximality of~$G$.
\end{proof}


\begin{lemma}\label{lem:crneighbor}
  Let~$uv$ and~$e$ be two independent edges in a $2$-plane drawing~$D$
  of a maximal $2$-planar graph~$G$ that cross at a point~$\chi$ so
  that the arc~$u\chi$ of~$uv$ is uncrossed in~$D$. Then for at least
  one endpoint~$p$ of~$e$ the arc~$p\chi$ is uncrossed in~$D$ and~$p$
  is adjacent to~$u$ in~$G$ via an uncrossed edge in~$D$.
\end{lemma}
\begin{proof}
  As~$e$ has at most two crossings in~$D$, at least one of the two
  arcs of~$e\setminus\chi$ is uncrossed. Let~$p$ be an endpoint of~$e$
  such that~$p\chi$ is uncrossed in~$D$. Then~$p$ and~$u$ appear on a
  common face in~$D$ and, therefore, they are adjacent via an
  uncrossed edge in~$D$ by \cref{lem:uncrossed_edge}.
\end{proof}

\begin{lemma}\label{lem:credge}
  Let~$uv$ be a doubly crossed edge in a crossing-minimal $2$-plane
  drawing~$D$ of a $2$-planar graph. Denote the crossings of~$uv$
  by~$\alpha$ and~$\beta$ such that~$\alpha$ appears before~$\beta$
  when traversing~$uv$ starting from~$u$. Then~$u$ and~$\beta$ do not
  appear on a common face in~$D$.
\end{lemma}
\begin{proof}
  Suppose for a contradiction that~$u$ and~$\beta$ appear on a common
  face~$f$ of~$D$. Then we can redraw the edge~$uv$ starting from~$u$
  to reach~$\beta$ through the interior of~$f$ and continue from there
  along its original drawing. This eliminates at least one crossing
  (of~$uv$ at~$\alpha$) and does not introduce any crossing. Thus, the
  resulting drawing is a $2$-plane drawing of the same graph with
  strictly fewer crossings than~$D$, in contradiction to the
  crossing-minimality of~$D$.
\end{proof}

\section{Proofs from \cref{sec:hermit}: Hermits}

\HermitUncrossed*
\begin{proof}\label{PHermitUncrossed}
  For a contradiction, assume that the edge~$hy$ is crossed
  in~$D$. Let~$\chi$ be the first crossing along~$hy$ while traversing
  it starting from~$h$, and let~$f$ be the crossing edge. By
  \cref{lem:crneighbor} at least one endpoint of~$f$ is adjacent
  to~$h$ in~$G$, denote this endpoint by~$u$. As~$D$ is simple, we
  have~$u\ne y$, which implies~$u=x$.

  We claim that the edge~$hx$ is uncrossed. Otherwise, it can be
  redrawn by following~$f$ from~$x$ to~$\chi$ and then following~$hy$
  from~$\chi$ to~$h$; see \cref{fig:degree2_fig1} for
  illustration. Such a redrawing would decrease the number of
  crossings, in contradiction to the crossing-minimality of~$D$. Thus,
  the edge~$hx$ is uncrossed.  But now the edge~$f$ can be redrawn to
  follow~$hx$ from~$x$ to~$h$, then follow~$hy$ from~$h$ to~$\chi$,
  and then follow its original drawing. This redrawing eliminates the
  crossing between~$f$ and~$hy$ at~$\chi$ and does not introduces any
  crossing, which is a contradiction to the crossing-minimality
  of~$D$. Therefore, both edges incident to~$h$ are uncrossed
  in~$D$. The statement concerning the edge~$xy$ then follows by
  \cref{lem:uncrossed_edge}.
\end{proof}

\begin{figure}[htbp]
  \centering
  \includegraphics[scale=1]{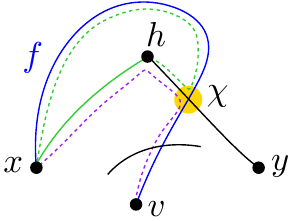}
  \caption{Edges incident to a hermit $h$ are uncrossed.}
  \label{fig:degree2_fig1}
\end{figure}

\OneHermitPerEdge*
\begin{proof}\label{POneHermitPerEdge}
  For a contradiction, assume that an edge~$xy$ of~$G$ hosts two
  hermits~$h,h'$. Then by \cref{lem:hermit_uncrossed} all four
  edges~$xh, yh, xh',yh'$ are uncrossed in~$D$. Then in~$D$ the
  vertex~$h'$ can be placed arbitrarily close to~$h$, and the
  edges~$xh'$ and~$yh'$ can be (re)drawn close to the edges~$xh$
  and~$yh$, respectively, so that an uncrossed edge can be drawn
  between~$h$ and~$h'$. But this is a contradiction to the maximality
  of~$G$. Thus, no such edge~$xy$ exists in~$G$.
\end{proof}

\EfficientHermitNeighbors*
\begin{proof}\label{PEfficientHermitNeighbors}
  By~\cref{lem:hermit_uncrossed}, the edges $hx, hy$ and $xy$ are
  uncrossed in~$D$; after redrawing---if necessary---the edge~$xy$ to
  closely follow the uncrossed path~$xhy$, we may assume that there is
  a triangular face~$\phi$ bounded by the cycle~$xyh$
  in~$D$. Let~$\psi\ne\phi$ denote the other face
  with~$h\in\partial\psi$; see \cref{fig:degree2_faces}~(left) for
  illustration. Since~$G$ is maximal $2$-planar and~$\deg(h)=2$, no
  other (than~$x,y,h$) vertex of~$G$ is on~$\partial\psi$: If there
  was such a vertex, then it would be neighbor of~$h$ by
  \cref{lem:uncrossed_edge}. As~$G$ has at least five vertices and it
  is $2$-connected by \cref{thm:2connected}, both~$x$ and~$y$ have at
  least one neighbor outside of~$x,y,h$. Assume for a contradiction
  that~$\deg(x)=3$, and let~$e=xa$ denote the edge between~$x$
  and~$a\notin\{y,h\}$.

  As~$a\notin\partial\psi$, the edge~$e$ is crossed in~$D$. Let~$f$ be
  the first edge that crosses~$e$ when traversing it starting
  from~$x$, and denote this crossing by~$\chi$. By
  \cref{lem:crneighbor} at least one endpoint~$q$ of~$f$ is adjacent
  to~$x$ in~$G$, that is, we have~$q\in\{a,y,h\}$. As~$D$ is simple
  and~$a$ is an endpoint of~$e$, which crosses~$f$, we have~$q\ne a$.
  As all edges at~$h$ are uncrossed in~$D$, whereas~$f$ is crossed, it
  follows that~$q\ne h$. Therefore, we have~$q=y$. As the part of~$f$
  between~$y$ and~$\chi$ is uncrossed, we can change the drawing as
  follows: Place~$x,h$ close to~$y$ on the side of~$f$ which the part
  of~$e$ between~$\chi$ and~$a$ attaches to. The triangle~$xyh$ can be
  drawn uncrossed, and the edge~$e$ can be drawn to follow~$f$ up
  to~$\chi$ and then continue along its original drawing. In this way,
  we eliminate the crossing at~$\chi$ without introducing any
  crossing, a contradiction to the crossing-minimality of~$D$. See
  \cref{fig:degree2_faces}~(right) for illustration.
  
  It follows that~$\deg(x)\ge 4$ and, by an analogous reasoning,
  also~$\deg(y)\ge 4$.
\end{proof}

\begin{figure}[htbp]
  \hfill\includegraphics[scale=1,page=1]{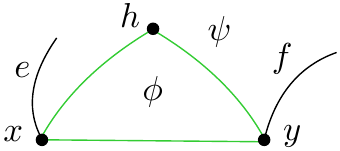}\hfill
  \includegraphics[scale=1,page=3]{degree2_faces}\hfill\hfill
  \caption{
    A hermit~$h$ and its neighborhood (left); redraw to obtain a
    contradiction (right).}
  \label{fig:degree2_faces}
\end{figure}

\HermitDeg*
\begin{proof}\label{PHermitDeg}
  By \cref{lem:efficient_hermit_neighbors} all neighbors of hermits
  have degree at least four. Hence, the statement holds for~$i\le 3$.
  Let~$v$ be a vertex that is adjacent to a hermit, and
  let~$H_v\ne\emptyset$ denote the set of hermits adjacent to~$v$
  in~$G$. Consider the circular order~$C_v$ of edges incident to~$v$
  in~$D$.

  We may assume that for every hermit~$h\in H_v$ the edge~$vh$ is
  adjacent in~$C_v$ to the edge that hosts~$h$: As long as this is not
  the case for some hermit~$h\in H_v$, we redraw the host edge
  sufficiently close to the path in~$D$ that is formed by the two
  edges incident to~$h$, which are uncrossed in~$D$ by
  \cref{lem:hermit_uncrossed}. As every edge hosts at most one hermit
  by \cref{lem:one_hermit_per_edge}, every edge is redrawn at most
  once during this process.

  Hence, each hermit~$h\in H_v$ corresponds to a pair of uncrossed
  edges~$vh,vx$ that are adjacent in~$C_v$. Note that we can adjust
  the order of~$vh,vx$ to our liking, by redrawing the uncrossed
  path~$vhx$ on either side of the edge~$vx$. We claim that both
  neighbors of~$vh,vx$ in~$C_v$ are crossed in~$D$. Suppose for a
  contradiction that, without loss of generality the edges~$vx,vh,vy$
  are consecutive in~$C_v$ and all uncrossed in~$D$. Then~$h$ and~$y$
  appear on a common face in~$D$ and, therefore, they are adjacent
  in~$G$ by \cref{lem:uncrossed_edge}. As~$x,y,h$ are pairwise
  distinct neighbors of~$v$ in~$G$, we have~$\deg(h)\ge 3$, in
  contradiction to our assumption that~$h$ is a hermit. Hence, the
  claim holds and, in particular, any two pairs of edges corresponding
  to hermits in~$C_v$ are separated by at least one crossed edge. So
  overall we have at least three edges incident to~$v$ for each hermit
  in~$H_v$, which proves the lemma.
\end{proof}

\DegreeFourHermit*
\begin{proof}\label{PDegreeFourHermit}
  By \cref{lem:hermit_uncrossed} we have an uncrossed triangle~$uvh$
  in~$D$. By redrawing---if necessary---the edge~$uv$ close to the
  path~$uhv$ we may assume that~$uvh$ bounds a face of~$D$. We can
  place~$h$ on either side of~$uv$ without changing anything else in
  the drawing but the orientation of the triangular face~$uvh$.

  We claim that both~$ux$ and~$uw$ are crossed in~$D$. Suppose for a
  contradiction that, for instance, the edge~$ux$ is uncrossed
  in~$D$. Then we place~$h$ so that it appears in between~$x$ and~$v$
  in the circular order of neighbors around~$u$. Then~$x$ and~$h$
  appear on a common face and by \cref{lem:uncrossed_edge} they are
  adjacent in~$G$, in contradiction to~$h$ being a hermit. This proves
  our claim that both~$ux$ and~$uw$ are crossed in~$D$.

  Let~$ab$ be the first edge that crosses~$ux$ when traversing~$ux$
  starting from~$u$, and denote this crossing by~$\alpha$. We
  have~$h\notin\{a,b\}$ because both edges incident to~$h$ are
  uncrossed in~$D$ whereas the edge~$ab$ is crossed. By
  \cref{lem:crneighbor} at least one endpoint of~$ab$ is adjacent
  to~$u$ via an uncrossed edge in~$D$; let~$a$ be such an endpoint.
  As the only neighbors of~$u$ via an uncrossed edge in~$D$ are the
  vertices~$v$ and~$h$ and we observed~$h\notin\{a,b\}$, it follows
  that~$a=v$.

  Next we claim that~$vb$ is doubly crossed in~$D$. Suppose for a
  contradiction that~$vb$ is singly crossed only. Then~$u$ and~$b$
  appear on a common face in~$D$, and so by \cref{lem:uncrossed_edge}
  there is an uncrossed edge~$ub$ in~$D$. The uncrossed neighbors
  of~$u$ in~$D$ are~$h$ and~$v$, but we know that~$b\notin\{h,v\}$, a
  contradiction. This proves our claim that~$vb$ is doubly crossed
  in~$D$. Denote the second (other than~$\alpha$) crossing of~$vb$
  by~$\beta$.

  Denote by~$P$ the uncrossed path~$\alpha uv$ in~$D$ (the first arc
  is along the edge~$xu$), and let~$P_\circ$ denote the uncrossed
  closed path~$\alpha uv\alpha$ in~$D$ (first arc along the edge~$xu$
  and last arc along the edge~$vb$). Place~$h$ on the side of~$uv$
  that puts it on the same side of~$P_\circ$ as~$w$. In other words,
  place~$h$ on the side of $uv$ such that~$v$ and~$x$ are consecutive 
  in the rotation of $u$. Then~$P$ appears
  along the boundary of a face~$f$ of~$D$.

  The edges~$vb$ and~$ux$ can cross in two possible ways: Either so
  that~$P_\circ$ separates~$b,x$ from~$w$ or so that all of~$b,w,x$
  are on the same side of~$P_\circ$; see
  \cref{fig:degree4_doubly_crossed:1} for illustration. In the former
  case the crossing~$\beta$ appears on~$\partial f$, together
  with~$v$, which is a contradiction to \cref{lem:credge}. It follows
  that we always are in the latter case and all of~$b,h,w,x$ are on
  the same side of~$P_\circ$. Then we
  have~$P_\circ\subseteq\partial f$ and, in fact, we claim
  that~$P_\circ=\partial f$. To see this observe that any vertex
  located inside~$f$ is not connected to~$u$, by assumption
  on~$\deg(u)$ and the location of~$h,w,x$ with respect to~$P_\circ$
  and~$f$. Thus, if any vertex is located inside~$f$, then~$v$ is a
  cut-vertex of~$G$, in contradiction to \cref{thm:2connected}. It
  follows that~$P_\circ=\partial f$, as claimed, and so the face~$f$
  is a triangle, as claimed in the statement of the lemma.

  \begin{figure}[htbp]
    \hfill\includegraphics[page=1]{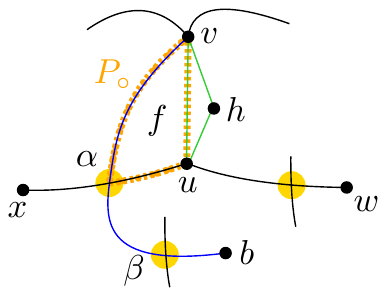}\hfill
    \includegraphics[page=2]{degree4_doubly_crossed}\hfill\hfill
    \caption{Proof of \cref{lem:degree4_hermit}: Two ways the
      edges~$ux$ and~$vb$ may cross.}
    \label{fig:degree4_doubly_crossed:1}
  \end{figure}
  
  By symmetry (between~$x$ and~$w$) there exists a doubly crossed
  edge~$vc$ in~$D$ that crosses~$uw$ at a crossing~$\gamma$ so
  that~$uv\gamma$ bounds a triangular face in~$D\setminus h$. Next, we
  show that both~$uw$ and~$ux$ are doubly crossed in~$D$.

  Assume for a contradiction that one of~$ux$ or~$uw$, say, the
  edge~$ux$ is singly crossed in~$D$. Consider the edge~$vc$, which
  crosses~$uw$ at~$\gamma$, let~$pq$ be the first edge that
  crosses~$vc$ when traversing~$vc$ starting from~$c$, and denote this
  crossing by~$\delta$. As~$vc$ is doubly crossed in~$D$, we
  have~$\delta\ne\gamma$. By \cref{lem:crneighbor} at least one
  endpoint of~$pq$ is adjacent to~$c$ via an uncrossed edge in~$D$,
  and further the arc of $pq$ between this endpoint and $\delta$ is uncrossed;
  let~$p$ be such an endpoint. We distinguish two cases, refer to
  \cref{fig:degree4_doubly_crossed:2} for illustration.

  \begin{figure}[htbp]
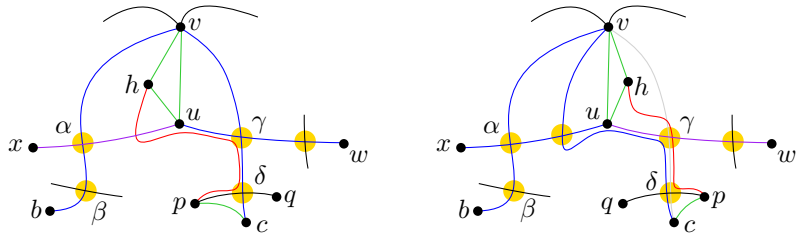

    \hfill\includegraphics[page=3]{degree4_doubly_crossed}\hfill
    \includegraphics[page=4]{degree4_doubly_crossed}\hfill\hfill
    \caption{Proof of \cref{lem:degree4_hermit}. Assuming~$ux$ is singly
      crossed only, there are two ways the edges~$vc$ and~$pq$ may
      cross. Either way, we can add a new edge (shown red) to arrive at
      a contradiction.}
    \label{fig:degree4_doubly_crossed:2}
  \end{figure}
  
  \begin{case}[The directed edges~$pq$ and~$uw$ cross~$vc$ from the
    same side in~$D$]
    Then~$u\gamma\delta p$ is an uncrossed path that appears along the
    boundary of a face in~$D$. We place~$h$ inside the triangular
    face~$uv\alpha$ of~$D$. Then we can add a singly crossed edge~$hp$
    by starting from~$h$, then crossing~$ux$ and then following the
    uncrossed path~$u\gamma\delta p$ to~$p$. Only one new crossing is
    created, between the new edge~$hp$ and the edge~$ux$, which is
    singly crossed only in~$D$. Thus, the resulting drawing is
    $2$-plane, in contradiction to the maximality of~$G$.
  \end{case}

  \begin{case}[The directed edges~$pq$ and~$uw$ cross~$vc$ from
    different sides in~$D$]
    We place~$h$ inside the triangular face~$vu\gamma$ of~$D$. Then we
    redraw the edge~$vc$ to start from~$v$ into the triangular
    face~$uv\alpha$, then cross~$ux$ and follow the uncrossed
    path~$u\gamma\delta$ to~$\delta$, and continue from there along
    its original drawing. Denote the resulting drawing by~$D'$. The
    edge~$vc$ is doubly crossed in both~$D$ and~$D'$, the edge~$ux$ is
    singly crossed in~$D$ and doubly crossed in~$D'$, and the
    edge~$uw$ is doubly crossed in~$D$ and singly crossed in~$D'$.  We
    can add a new edge~$hp$ to~$D'$ by starting from~$h$, then
    cross~$uw$ at~$\gamma$, continue along the drawing of~$vc$ in~$D$
    to~$\delta$, and then along the uncrossed arc~$\delta p$ of~$pq$
    to~$p$. The new edge~$hp$ is singly crossed and~$uw$ is doubly
    crossed, so the resulting drawing is $2$-plane, in contradiction
    to the maximality of~$G$.
  \end{case}
  
  In both cases we arrive at a contradiction. It follows that
  both~$ux$ and~$uw$ are doubly crossed in~$D$, as claimed in the
  statement of the lemma.

  It remains to prove the statements concerning the degree of~$v$ and
  the number of hermits adjacent to~$v$ in case~$\deg(v)=6$.

  We have identified four neighbors of~$v$ in~$G$ already: the
  vertices~$h,u,b,c$. If these are the only neighbors of~$v$, then we
  can place~$v$ and~$h$ on the other side of the uncrossed
  path~$\alpha u\gamma$ and redraw the edges~$vb$ and~$vc$ accordingly
  without crossing~$ux$ and~$uw$, respectively; see
  \cref{fig:degree4_hermit_nodeg5}~(left) for illustration. The
  resulting drawing of~$G$ is $2$-plane and has fewer crossings
  than~$D$, a contradiction. Therefore, we have~$\deg(v)\ge 5$. Denote
  the fifth neighbor (other than~$h,u,b,c$) of~$v$ in~$G$ by~$y$. 

  \begin{figure}[htbp]
    \centering
    \begin{minipage}[t]{.48\linewidth}
      \centering\includegraphics[page=4]{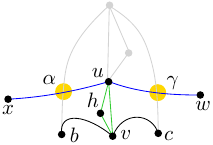}
    \end{minipage}
    \begin{minipage}[t]{.48\linewidth}
      \centering\includegraphics[page=3]{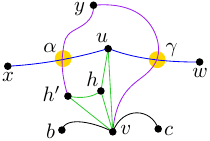}
    \end{minipage}
    \caption{Redrawings to obtain a contradiction:
      if~$\deg(v)=4$~(left) and if~$\deg(v)=6$ and~$v$ is adjacent to
      two hermits~(right).}
    \label{fig:degree4_hermit_nodeg5}
  \end{figure}
  
  Suppose for a contradiction that~$\deg(v)=5$. If $vy$ was uncrossed, 
  we can again redraw similar to \cref{fig:degree4_hermit_nodeg5}~(left)
  and redraw edge $vy$ to cross $ux$ or $uw$ at $\alpha$ or $\gamma$,
  resulting in a $2$-plane drawing of~$G$ with fewer crossings than $D$, which is not possible.
  Thus, $vy$ must be crossed.
  Let~$pq$ be the first
  edge that crosses~$vy$ when traversing~$vy$ starting from~$v$, and
  denote this crossing by~$\chi$. By \cref{lem:crneighbor} for at
  least one endpoint of~$pq$, suppose without loss of generality it
  holds for~$p$, the arc~$p\chi$ is uncrossed in~$D$ and~$pv$ is an
  uncrossed edge of~$D$. As the only neighbors of~$v$ along an
  uncrossed edge in~$D$ are~$h$ and~$u$ and both edges incident to~$h$
  are uncrossed (whereas~$pq$ is crossed), it follows that~$p=u$,
  which implies~$q\in\{w,x\}$. However, the first crossing of~$ux$
  when traversing~$ux$ starting from~$u$ is with~$vb$
  at~$\alpha\ne\chi$, and the first crossing of~$uw$ when
  traversing~$uw$ starting from~$u$ is with~$vc$
  at~$\gamma\ne\chi$. Hence, neither~$uw$ nor~$ux$ contain an
  uncrossed arc~$u\chi$, a contradiction. It follows
  that~$\deg(v)\ge 6$, as claimed in the statement of the lemma.

  Finally, suppose for a contradiction that~$\deg(v)=6$ and that~$v$
  is adjacent to two hermits in~$G$. Denote the second (other
  than~$h$) hermit adjacent to~$v$ by~$h'$, and let~$vy$ be the edge
  of~$G$ that hosts~$h'$. Then we can apply a similar redrawing as
  before, by placing both~$v$ and~$h$ on the other side of the
  uncrossed path~$\alpha u\gamma$ and let the two edges~$h'y$ and~$vy$
  cross~$uw$ and~$ux$ at~$\alpha$ and~$\gamma$, respectively; see
  \cref{fig:degree4_hermit_nodeg5}~(right) for illustration. The
  resulting drawing is $2$-plane and we can add an uncrossed
  edge~$hh'$, in contradiction to the maximality of~$G$. Therefore,
  if~$\deg(v)=6$, then~$v$ is adjacent to at most one hermit in~$G$,
  as claimed.
\end{proof}

\section{Proofs from \cref{sec:degreethree}: Degree-three vertices}

\subsection{T3-1 vertices}

\DegreeThreeTwoCrossed*
\begin{proof}\label{PDegreeThreeTwoCrossed}
  As~$u$ is a T3-1 vertex, the edge~$uv$ is the only uncrossed edge
  incident to~$u$ in~$D$ and the edges~$ux$ and~$uw$ are both crossed
  in~$D$.

  Let~$ab$ be the first edge that crosses~$ux$ when traversing~$ux$
  starting from~$u$, and denote this crossing by~$\alpha$. By
  \cref{lem:crneighbor} for at least one endpoint of~$ab$, suppose
  without loss of generality it holds for~$a$, the arc~$a\alpha$ is
  uncrossed in~$D$ and~$ua$ is an uncrossed edge of~$D$. As the only
  neighbor of~$u$ via an uncrossed edge in~$D$ is~$v$, it follows
  that~$a=v$.

  Next we claim that~$vb$ is doubly crossed in~$D$. Suppose for a
  contradiction that~$vb$ is singly crossed only. Then~$u$ and~$b$
  appear on a common face in~$D$, and so by \cref{lem:uncrossed_edge}
  there is an uncrossed edge~$ub$ in~$D$. As the only neighbor of~$u$
  via an uncrossed edge in~$D$ is~$v$, it follows that~$b=v$, in
  contradiction to~$G$ being a simple graph (without loops). Hence,
  the edge~$vb$ is doubly crossed in~$D$, as claimed. Denote the
  second (other than~$\alpha$) crossing of~$vb$ by~$\beta$.

  Let~$P_\circ$ denote the uncrossed closed path~$\alpha u v\alpha$
  in~$D$, where the first arc is along~$xu$ and the last arc is
  along~$vb$. The edge~$vb$ can cross~$ux$ in two possible ways:
  either so that~$P_\circ$ separates~$x$ and~$w$ or so that~$x$
  and~$w$ are on the same side of~$P_\circ$; see
  \cref{fig:degree3_doubly_crossed:1} for illustration. In the former
  case the crossing~$\beta$ appears on a common face with~$v$, in
  contradiction to \cref{lem:credge}. It follows that we always are in
  the latter case and all of~$b,w,x$ are on the same side
  of~$P_\circ$. Further, no vertex adjacent to $v$ can lie in the face
  of~$D$ bounded by~$P_\circ$ since $G$ is $2$-connected. 
  In particular, the face of~$D$ bounded by~$P_\circ$ is
  a triangle, which is incident to~$uv$ and bounded by (parts of)
  edges incident to~$u$ and doubly crossed edge incident to~$v$, as
  claimed in the statement of the lemma.

  \begin{figure}[htbp]
    \hfill\includegraphics[page=1]{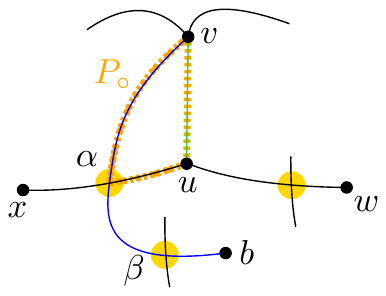}\hfill
    \includegraphics[page=2]{degree3_doubly_crossed}\hfill\hfill
    \caption{Proof of \cref{lem:degree3_2crossed_neighbordeg5}: Two
      ways the edges~$ux$ and~$vb$ may cross.}
    \label{fig:degree3_doubly_crossed:1}
  \end{figure}

  By symmetry (between~$x$ and~$w$) there exists a doubly crossed
  edge~$vc$ in~$D$ that crosses~$uw$ at a crossing~$\gamma$ so
  that~$uv\gamma$ bounds a triangular face in~$D$. 

  It remains to show that~$\deg(v)\ge 5$. We have identified three
  neighbors of~$v$ in~$G$ already: the vertices~$u,b,c$. If these are
  the only neighbors of~$v$, then we can place~$v$ on the other side
  of the uncrossed path~$\alpha u\gamma$ and redraw the edges~$vb$
  and~$vc$ accordingly without crossing~$ux$ and~$uw$, respectively;
  see \cref{fig:degree3_nodeg5}~(left) for illustration. The resulting
  drawing of~$G$ is $2$-plane and has fewer crossings than~$D$, a
  contradiction. Therefore, we have~$\deg(v)\ge 4$. Denote the fourth
  neighbor (other than~$u,b,c$) of~$v$ in~$G$ by~$y$. 

  \begin{figure}[htbp]
    \centering
    \begin{minipage}[t]{.48\linewidth}
      \centering\includegraphics[page=3]{degree3_1}
    \end{minipage}
    \begin{minipage}[t]{.48\linewidth}
      \centering\includegraphics[page=1]{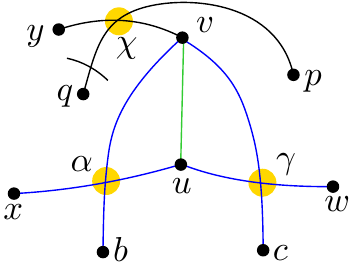}
    \end{minipage}
    \caption{Redrawing to obtain a contradiction
      if~$\deg(v)=3$~(left); if~$\deg(v)=4$, then~$v$ and~$p$ appear
      on a common face~(right).}
    \label{fig:degree3_nodeg5}
  \end{figure}
  
  Suppose for a contradiction that $\deg(v)=4$. 
  If $vy$ was uncrossed, we can again redraw similar to \cref{fig:degree3_nodeg5}~(left)
  and redraw edge $vy$ to cross $ux$ or $uw$ at $\alpha$ or $\gamma$,
  resulting in a $2$-plane drawing of~$G$ with fewer crossings than $D$, which is not possible.
  Thus, $vy$ must be crossed.Let~$pq$ be the first edge that crosses~$vy$ when traversing~$vy$
  starting from~$v$, and denote this crossing by~$\chi$; see
  \cref{fig:degree3_nodeg5}~(right) for illustration. By
  \cref{lem:crneighbor} for at least one endpoint of~$pq$, suppose
  without loss of generality it is~$p$, the arc~$p\chi$ is uncrossed
  in~$D$ and~$pv$ is an uncrossed edge of~$D$. As the only neighbor
  of~$v$ along an uncrossed edge in~$D$ is~$u$, it follows that~$p=u$,
  which implies~$q\in\{w,x\}$. However, the first crossing of~$ux$
  when traversing~$ux$ starting from~$u$ is with~$vb$
  at~$\alpha\ne\chi$, and the first crossing of~$uw$ when
  traversing~$uw$ starting from~$u$ is with~$vc$
  at~$\gamma\ne\chi$. Hence, neither~$uw$ nor~$ux$ contain an
  uncrossed arc~$u\chi$, a contradiction. It follows
  that~$\deg(v)\ge 5$, as claimed in the statement of the lemma.
\end{proof}

\subsection{T3-2 vertices}

\CRadJ*
\begin{proof}\label{PCRadJ}
  Let~$u$ be a T3-2 vertex in~$D$ such that~$uv$ is the unique
  incident edge that is crossed in~$D$. Denote the other two neighbors
  of~$u$ in~$G$ by~$w$ and~$x$. Let~$e=ab$ be the first edge that
  crosses~$uv$ when traversing it starting from~$u$, and denote this
  crossing by~$\alpha$. By \cref{lem:crneighbor} at least one endpoint
  of~$e$ is adjacent to~$u$ in~$D$ via an uncrossed edge; let~$a$ be
  such a neighbor. By assumption there are exactly two uncrossed edges
  incident to~$u$ in~$D$, which implies~$a\in\{w,x\}$. Suppose without
  loss of generality that~$a=w$.

  First we show that~$uv$ is singly crossed in~$D$. Suppose for a
  contradiction that~$uv$ is doubly-crossed. We devise another
  crossing-minimal $2$-plane drawing~$D'$ of~$G$ with strictly fewer
  doubly crossed edges than~$D$, in contradiction to~$D$ being
  admissible (i.e., minimizing the number of doubly crossed edges
  among all crossing-minimal $2$-plane drawings of~$G$).

  Observe that we can redraw the edge~$e$ starting from~$w$ to
  approach~$\alpha$ from either side of~$uv$: One side can be reached
  by following the uncrossed edge~$wu$ towards~$u$ and then the
  uncrossed arc~$u\alpha$ of~$uv$ to~$\alpha$; the other side can be
  reached by following the uncrossed path~$wux$, then crossing~$uw$
  and following the uncrossed arc~$u\alpha$ of~$uv$ to~$\alpha$. So to
  obtain~$D'$ we redraw~$e$ to approach~$\alpha$ on the side from
  where~$e$ can continue along its original drawing to~$b$ without
  crossing~$uv$; see \cref{fig:Xi_prop2} for illustration. This
  eliminates a crossing on~$uv$, making~$uv$ a singly crossed
  edge. The number of crossings along~$e$ does not increase, as one
  crossing (at~$\alpha$) is eliminated and at most one other crossing
  (with~$ux$) is introduced. The edge~$ux$ is uncrossed in~$D$, so it
  is at most singly crossed in the new drawing~$D'$. The number of
  crossings along all other edges remains unchanged, and the overall
  number of crossings does not increase. Hence, the drawing~$D'$ is
  $2$-plane, crossing-minimal, and the number of doubly crossed edges
  is strictly smaller than in~$D$, a contradiction. It follows
  that~$uv$ is singly crossed, that is, the crossing~$\alpha$
  with~$wb$ is its only crossing in~$D$.

  \begin{figure}[htbp]
    \centering
    \includegraphics[scale=1]{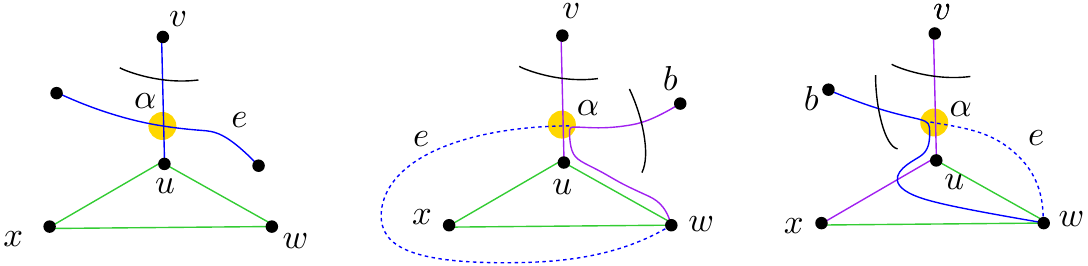}
    \caption{Redrawing the edge $e$ so as to decrease the number of doubly crossed edges.}
    \label{fig:Xi_prop2}
  \end{figure}
  
  The arc~$w\alpha$ of~$wb$ is uncrossed by \cref{lem:crneighbor}. So
  it remains to show that~$wb$ is doubly crossed in~$D$. Suppose for a
  contradiction that~$wb$ is singly crossed, that is, its only
  crossing in~$D$ is~$\alpha$. Then~$b$ and~$u$ are on a common face
  in~$D$. So by \cref{lem:uncrossed_edge} there is an uncrossed
  edge~$ub$ in~$D$, which implies that~$b\in\{w,x\}$. As the graph~$G$
  is simple (no loops), we have~$b=x$. However, as~$wux$ is an
  uncrossed path in~$D$ and~$\deg(u)=3$, the vertices~$x$ and~$w$ lie
  on a common face and, therefore, they are connected by an uncrossed
  edge in~$D$. In particular, the edge~$xw$ does not cross~$uv$
  in~$D$, a contradiction. It follows that~$wb$ is doubly crossed
  in~$D$, as claimed.
\end{proof}

\TTTChoose*
\begin{proof}\label{PTTTChoose}
  Obtain~$D'$ from~$D$ by redrawing the edge~$wb$ as follows:
  Leave from~$w$ on the other side of~$wu$, then follow~$wu$ to~$u$,
  cross~$ux$, then follow~$uv$ up to the crossing with the old drawing
  of~$wb$, and proceed from there along its original drawing
  to~$b$. For every edge the number of crossings remains the same,
  except for~$uv$ (singly crossed in~$D$ and uncrossed in~$D'$)
  and~$ux$ (uncrossed in~$D$ and singly crossed in~$D'$). Therefore,
  the drawing~$D'$ is admissible and it has the two listed properties.
\end{proof}

\subsection{Proof of~\cref{lem:degree3_1crossed_degrees555}}

As in the statement of the lemma, let~$u$ be a T3-2 vertex in~$D$, and
let~$v,w,x$ be the neighbors of~$u$ such that the edge~$uv$ is singly
crossed by a doubly crossed edge~$wb$, and the edges~$ux$ and~$uw$ are
uncrossed in~$D$. Denote the crossing of~$uv$ and~$wb$ by~$\alpha$,
and denote the second crossing of~$wb$ by~$\beta$.

The path~$xuw$ is uncrossed and part of the boundary of a face in~$D$,
so by \cref{lem:uncrossed_edge} there is an uncrossed edge~$wx$
in~$D$. Similarly, the path~$wu\alpha$ is uncrossed and part of the
boundary of a face~$f$ in~$D$. Thus, by \cref{lem:credge} the
edge~$wb$ is uncrossed between~$w$ and~$\alpha$, whereas~$w$
and~$\beta$ do not appear on a common face in~$D$. As~$uv$ has a
single crossing only, at~$\alpha$, it follows that~$w$ and~$v$ are on
a common face in~$D$, and, therefore, they are adjacent in~$G$ by
\cref{lem:uncrossed_edge}. So we are facing a drawing as depicted in
\cref{fig:wdeg5}~(left), not as in \cref{fig:wdeg5}~(middle),
where~$w$ and~$\beta$ appear on a common face.\medskip

%

\begin{lemma}
  We have~$\deg(w)\ge 4$.
\end{lemma}
\begin{proof}
  We have established all of~$u,x,v,b$ to be neighbors of~$w$
  in~$G$. It remains to argue that they are pairwise
  distinct. For~$u,x,w$ this is the case by assumption (that~$v,x,w$
  are the three neighbors of~$u$ in~$G$). As~$uv$ crosses~$wb$, we
  have~$b\notin\{u,v\}$ because~$D$ is simple. As the edge~$wx$ is
  uncrossed but the edge~$wb$ is crossed, we also have~$b\ne x$.
\end{proof}


\begin{figure}[htbp]
  \centering\includegraphics[page=1]{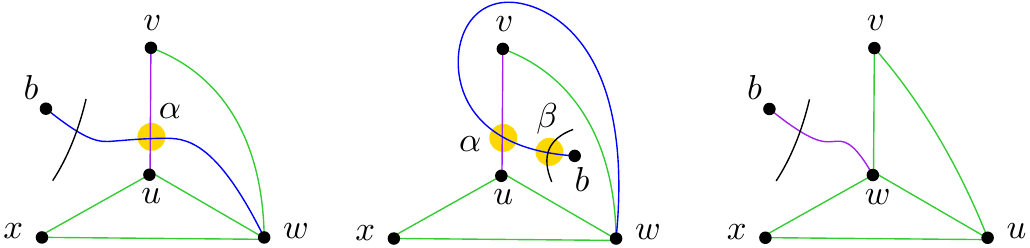}
  \caption{The situation around a T3-2 vertex~$u$~(left). A drawing
    that is not admissible because~$w$ and~$\beta$ appear on a common
    face~(middle). Removing a crossing if~$\deg(w)=4$~(right).}
  \label{fig:wdeg5}
\end{figure}

\DegreeThreeOneCrossed*
\begin{proof}\label{PDegreeThreeTwoOneCrossed}
  We know that~$w$ is adjacent to all of~$u, v, x$ and~$b$
  in~$G$. If~$\deg(w)=4$, then we can exchange the position of~$u$
  and~$w$ in the drawing; see \cref{fig:wdeg5}~(right). This
  eliminates the crossing between~$uv$ and~$wb$ at~$\alpha$ without
  introducing any crossing, which is a contradiction to the crossing
  minimality of~$D$. Thus, we have~$\deg(w)\ge 5$.

  Observe that the local roles of~$v$ and~$x$ are symmetric, as,
  by Lemma~\ref{prop:t32-choose}, the
  edge~$wb$ can be drawn to cross any one of~$uv$ or~$ux$ without
  changing anything else. Therefore it suffices to show
  that~$\deg(v)\ge 4$ and~$\deg(x)\ge 4$ follows by symmetry. We
  have~$\deg(v)\ge 3$ because there is a crossed edge incident to~$v$
  in~$D$ and both edges incident to a hermit are uncrossed by
  \cref{lem:hermit_uncrossed}. Denote the third neighbor (other
  than~$u,w$) of~$v$ in~$G$ by~$y$.

  Now assume for a contradiction that~$\deg(v)=3$. Then~$y\ne x$
  because if~$y=x$ we could redraw~$v$ close to~$u$, in the face
  bounded by the uncrossed edges~$uw$ and~$ux$, so that all edges
  incident to~$v$ are uncrossed without introducing any crossing. At
  least one crossing, the one between~$uv$ and~$wb$, is eliminated
  this way, in contradiction to the crossing minimality
  of~$D$. Therefore, we have~$y\ne x$.

  We claim that the edge~$vy$ is crossed. If~$vy$ is uncrossed, then
  we can redraw the edge~$wb$ to cross~$ux$ rather than~$uv$ and then
  draw an uncrossed edge from~$u$ along the uncrossed edges~$uv$
  and~$vy$ to~$y$. This is a contradiction to the maximality of~$G$
  and, therefore, the edge~$vy$ is crossed, as claimed.

  Let~$st$ be the first edge that crosses~$vy$ when traversing~$vy$
  starting from~$v$, and denote this crossing by~$\sigma$. By
  \cref{lem:crneighbor} at least one endpoint of~$st$ is connected by
  an uncrossed edge to~$v$ in~$D$; let~$s$ be such an
  endpoint. Then~$s\in\{u,w\}$ and as we know all edges incident
  to~$u$ and none of them crosses~$vy$ in~$D$, we have~$s=w$. 

  Clearly, we have~$t\ne u$ because the edge~$wu$ is uncrossed
  and~$t\ne y$ because~$wt$ crosses~$vy$ and~$D$ is simple.  We claim
  that the edge~$wt$ is doubly crossed. Suppose that~$wt$ is singly
  crossed. Then~$v$ and~$t$ appear on a common face in~$D$ and,
  therefore, they are adjacent in~$G$, in contradiction
  to~$\deg(v)=3$. Thus, the edge~$wt$ is doubly crossed, as
  claimed. Denote its second (other than~$\sigma$) crossing by~$\tau$.

  By \cref{lem:credge} we know that~$w$ and~$\tau$ are not on a common
  face in~$D$, and so~$v$ and~$\tau$ \emph{are} on a common face~$f$
  in~$D$. Redraw the edge~$wb$ to cross~$ux$ instead of~$uv$, and
  redraw the edge~$wt$ starting from~$w$ to follow~$wu$, then
  cross~$uv$, and then proceed across~$f$ to~$\tau$, and from there on
  along its original drawing; see \cref{fig:wdeg52}~(middle). The
  resulting drawing~$D'$ has exactly the same number of crossings
  as~$D$. The number of crossings per edge change only for~$ux$, which
  is uncrossed in~$D$ and singly crossed in~$D'$, and~$vy$, which has
  one fewer crossing in~$D'$ than in~$D$. So~$D'$ is a
  crossing-minimal $2$-plane drawing of~$G$.

  If~$vy$ is doubly crossed in~$D$, then~$D'$ has strictly fewer
  doubly crossed edges than~$D$, a contradiction to~$D$ being
  admissible. It follows that~$vy$ is singly crossed in~$D$ and
  uncrossed in~$D'$. But then we can redraw~$wt$ to cross~$ux$ instead
  of~$uv$ and add an uncrossed edge~$uy$, which, noting
  that~$y\notin\{w,v,x\}$, is a contradiction to the maximality
  of~$G$; see \cref{fig:wdeg52}~(right). We conclude
  that~$\deg(v)\ge 4$, as claimed.
\end{proof}

\begin{figure}[htbp]
  \centering\includegraphics[page=2]{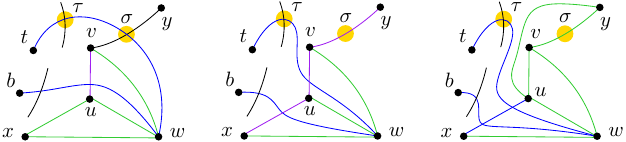}
  \caption{The situation around a T3-2 vertex~$u$ with a degree-three
    neighbor~$v$~(left). If the edge~$vy$ is doubly crossed in~$D$, a
    redrawing of~$wt$ yields an admissible drawing with fewer doubly
    crossed edges~(middle). Otherwise, the edge~$vy$ is singly crossed
    in~$D$, and a redrawing of~$wt$ makes~$vy$ uncrossed, which allows
    to add an edge~$uy$~(right).}
  \label{fig:wdeg52}
\end{figure}

\subsection{T3-3 hermits}

\DegThreeHermit*
\begin{proof}\label{PDegThreeHermit}
  By \cref{lem:efficient_hermit_neighbors} both neighbors of a hermit
  have degree at least four. It follows that~$\deg(z')=3$. As~$z$ is a
  T3-3 vertex, the edge~$zz'$ is uncrossed in~$D$. The neighbor of a
  T3-1 vertex along the (only) uncrossed edge has degree at least five
  by \cref{lem:degree3_2crossed_neighbordeg5}. Thus, we know that~$z'$
  is not of type T3-1. As by \cref{lem:degree3_1crossed_degrees555}
  all neighbors of a T3-2 vertex have degree at least four, we
  conclude that~$z'$ is not of type T3-2, either. The only remaining
  option is that~$z'$ is a T3-3 vertex, which forms an inefficient
  hermit with~$z$.
\end{proof}

\IneffHermitThree*
\begin{proof}\label{PIneffHermitThree}
  As~$xzy$ is an uncrossed path in~$D$, by \cref{lem:uncrossed_edge}
  we have an uncrossed edge~$xy$ in~$D$. Assume without loss of
  generality that the cycle~$xzy$ bounds a face in~$D$ (if necessary
  redraw the edge~$xy$ to closely follow the path~$xzy$). As~$G$ has
  at least five vertices and it is $2$-connected by
  \cref{thm:2connected}, both~$x$ and~$y$ have at least one neighbor
  outside of~$x,y,z,z'$. Thus, there is a (unique) face~$\psi$ in~$D$
  that contains the path~$xz'y$ on its boundary, along with parts of
  at least two more edges: an edge~$e$ incident to~$x$ and an edge~$f$
  incident to~$y$, so that~$e,f,xy$ are pairwise distinct. Similarly,
  there is a (unique) face~$\psi'$ in~$D$ that contains the edge~$xy$
  on its boundary but neither of the vertices~$z$ and~$z'$. Let~$e'$
  and~$f'$ denote the edge on~$\partial\psi'$ incident to~$x$ and~$y$,
  respectively, that is not the common edge~$xy$. See
  \cref{fig:degree3_faces}~(left) for illustration.
  
  \begin{figure}[htbp]
    \hfill\includegraphics[scale=1,page=1]{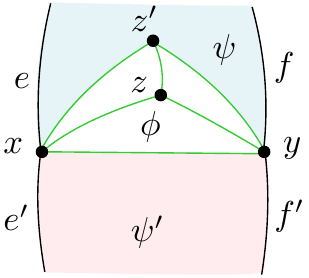}\hfill
    \includegraphics[scale=1,page=2]{degree3_faces}\hfill
    \includegraphics[scale=1,page=3]{degree3_faces}\hfill\hfill
    \caption{Proof of \cref{lem:deg3hermit_neighbors}: An inefficient
      hermit~$z,z'$ and its neighborhood (left); suppose that~$e=e'$
      (middle); redraw to obtain a contradiction (right).}
    \label{fig:degree3_faces}
  \end{figure}
  
  Since~$G$ is maximal $2$-planar and~$z'$ has a degree of exactly
  three, no other (than~$x,y,z'$) vertex of~$G$ is on~$\partial\psi$:
  If there was such a vertex, then it would be neighbor of~$z'$ by
  \cref{lem:uncrossed_edge}. Analogously, no other (than~$x,y$) vertex
  of~$G$ is on~$\partial\psi'$ because we could easily change the
  drawing by moving both~$z$ and~$z'$ to the other side of~$xy$. In
  particular, it follows that all of~$e,f,e',f'$ are crossed in~$D$.
  
  For a contradiction, suppose that~$\deg(x)=4$, which
  implies~$e = e'$. Let~$g$ denote the first edge that crosses~$e$
  when traversing it starting from~$x$, and denote this crossing
  by~$\chi$. As~$g$ has at most two crossings in~$D$, the arc of~$g$
  on one side of~$\chi$ is uncrossed in~$D$. As~$G$ is maximal
  $2$-planar, by \cref{lem:uncrossed_edge} the endpoint of~$g$ along
  this uncrossed arc is a neighbor of~$x$, which implies that this
  endpoint is~$y$.

  The simple closed curve~$C=x\chi y$ formed by the edge~$xy$ together
  with the corresponding arcs of~$e$ and~$g$ is uncrossed
  in~$D$. Denote~$V^-=V(D)\setminus\{x,y,z,z'\}$. We claim
  that all vertices of~$V^-$ lie on the same side of~$C$. To see this
  observe that~$z$ and~$z'$ have no neighbor in~$V^-$ and~$x$ has
  exactly one neighbor in~$V^-$. Hence, if there are vertices of~$V^-$
  on both sides of~$C$, then removal of~$y$ separates~$G$, in
  contradiction to~$G$ being $2$-connected. This proves our claim that
  all vertices of~$V^-$ lie on the same side of~$C$. In particular, it
  follows that~$g\in\{f,f'\}$.

  Observe that the situation is symmetric between~$f$ and~$f'$:
  If~$f=f'$ this is vacuously true; otherwise, we may change the
  drawing so that both~$z$ and~$z'$ appear on the other side of~$xy$,
  effectively exchanging the roles of~$\psi$ and~$\psi'$ as well as
  the roles of~$f$ and~$f'$. Hence, we may assume without loss of
  generality that~$g=f$. See \cref{fig:degree3_faces}~(middle) for
  illustration.

  Then we can redraw~$G$ as follows: Remove all of~$x,z,z'$ and place
  them in the face~$\psi''$ on the other side of~$f$. The $K_4$
  subdrawing formed by~$x,y,z,z'$ can be drawn without crossings
  inside~$\psi''$ and the edge~$e$ is drawn by following~$f$ up
  to~$\chi$ and then along the original drawing of~$e$ in~$D$. See
  \cref{fig:degree3_faces}~(right) for illustration. No crossing is
  introduced by these changes, and the crossing~$\chi$ is
  removed. Hence, the resulting drawing is a $2$-plane drawing of~$G$
  with strictly fewer crossings than~$D$, a contradiction to the
  crossing minimality of~$D$.

  It follows that~$\deg(e)\ge 5$ and, analogously,
  also~$\deg(f)\ge 5$.
\end{proof}

\subsection{T3-3 minglers}


\DegreeThreeUncrossed*
\begin{proof}\label{PDegreeThreeUncrossed}
  Denote by~$Q$ the~$K_4$-subdrawing induced by~$u,v,w,x$ in~$D$.
  Assume without loss of generality that all of~$xuv$, $uvw$,
  and~$xuw$ are triangular faces of~$D$. If, say, the triangle~$xuv$
  of~$G$ does not bound a face in~$D$, then redraw the edge~$xw$ to
  closely follow the uncrossed path~$xuw$ in~$D$.


  By \cref{lem:deg3_inefficient_hermit} all of~$v,w,x$ have degree at
  least four. For a contradiction, assume the degree of all three
  vertices is exactly four.


  Let~$e$, $f$, and~$g$ be the (unique) edges incident to~$v$, $w$,
  and~$x$, respectively, that are not part of~$Q$. We claim that no
  two of~$e,f,g$ cross. To see this, suppose for a contradiction
  that~$e$ and~$f$ cross at a point~$\chi$. Then we can exchange the
  position of~$v$ and~$w$ in~$D$ and redraw the edges~$e,f$
  accordingly, so that~$e$ follows the original drawing of~$f$ and
  vice versa up to~$\chi$, from where on both edges continue along
  their own original drawing. This change eliminates one crossing
  ($\chi$) and maintains the number of crossings per edge for all
  edges other than~$e,f$. As the original drawing~$D$ is $2$-plane, at
  most one of the two arcs of~$e\setminus\chi$ has at most one
  crossing. Analogously, at most one of the two arcs
  of~$f\setminus\chi$ has at most one crossing. Therefore, the new
  drawings of both~$e$ and~$f$ have at most two crossings each. (In
  fact, if one of them has two, then the other has none.) So the
  modified drawing is $2$-plane and has fewer crossings than~$D$, in
  contradiction to the crossing-minimality of~$D$. This proves
  that~$e$ and~$f$ do not cross and, by symmetry, our claim that no
  two of the edges~$e,f,g$ cross.

  Still, we claim that each of~$e,f,g$ is crossed in~$D$. Consider,
  for instance, the edge~$e$, and let~$z\notin\{u,v,w,x\}$ denote its
  other endpoint. If~$e$ is uncrossed in~$D$, then we can add an
  edge~$uz$ that leaves~$u$ within the triangle~$uvw$, then
  crosses~$vw$ and follows~$e$ to~$z$. In the resulting drawing~$D'$
  both~$vw$ and~$uz$ have one crossing, and all other crossings are
  the same as in~$D$. Thus, the drawing~$D'$ is $2$-plane and the
  graph~$G\cup uz$ is $2$-planar, in contradiction to the maximality
  of~$G$. It follows that~$e$ and by symmetry also~$f$ and~$g$ are
  crossed in~$D$.

  Let~$h=ab$ be the first edge that crosses~$e$ when traversing~$e$
  starting from~$v$. By \cref{lem:crneighbor} at least one endpoint
  of~$h$ is adjacent to~$v$ in~$G$, assume without loss of generality
  that~$a$ is such an endpoint. It follows that~$a\in\{v,w\}$, that
  is, we have~$h\in\{f,g\}$. But we have shown above that~$e$ does not
  cross any of~$f$ or~$g$, a contradiction. Therefore, at least one
  of~$v,w,x$ has degree at least five.

  Without loss of generality assume that~$v$ has degree at least five.
  Suppose for a contradiction that~$\deg(v)=5$
  and~$\deg(w)=\deg(x)=4$. Let~$e_1$ and~$e_2$ denote the edge
  of~$D\setminus Q$ incident to~$v$ that is adjacent to~$vw$ and~$vx$,
  respectively, in the circular order of incident edges around~$v$
  in~$D$. As above, let $f$, and~$g$ be the (unique) edges incident
  to~$w$ and~$x$, respectively, that are not part of~$Q$. Using the
  same argument as above for~$e,f,g$ we see that all of~$e_1,e_2,f,g$
  are crossed in~$D$ and that~$f$ and~$g$ do not cross.

  Let~$h_1=ab$ be the first edge that crosses~$f$ when traversing~$f$
  starting from~$w$. By \cref{lem:crneighbor} at least one endpoint
  of~$h$ is adjacent to~$w$ in~$G$, assume without loss of generality
  that~$a$ is such an endpoint. It follows that~$a\in\{v,x\}$ and
  given that~$f$ and~$g$ do not cross, we have~$a=v$ and~$f$ is the
  first edge that crosses~$h_1$ when traversing~$h_1$ starting
  from~$v$. Analogously we conclude that the first edge~$h_2$ that
  crosses~$g$ when traversing it starting from~$x$ is incident to~$v$
  and~$g$ is the first edge that crosses~$h_2$ when traversing~$h_2$
  starting from~$v$. It follows that~$h_1=e_1$ and~$h_2=e_2$ and that
  the local drawing looks as shown in
  \cref{fig:degree3_uncrossed_fig1}. Then we can flip~$u$ and~$v$ to
  the other side of the edge~$xw$ and redraw the edges~$e_1$ and~$e_2$
  to follow~$f$ and~$g$, respectively, up to their original crossings
  and then continue along their original drawing. The resulting
  drawing is $2$-plane and has two fewer crossings than~$D$, a
  contradiction to the crossing minimality of~$D$.

  Therefore, one of~$w,x$ has degree at least five or~$\deg(v)\ge 6$.
\end{proof}

\begin{figure}[htbp]
  \hfill 
  \centering
  \subfigure[Original drawing.]{
    \includegraphics[page=1]{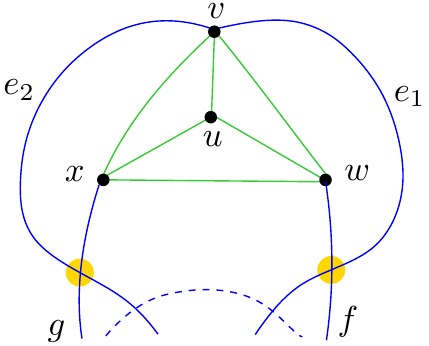}
    \label{fig:degree3_uncrossed_fig1}}
  \hfill
  \subfigure[Redrawing $u,v$
      and~$e_1,e_2$.]{
    \includegraphics[page=1]{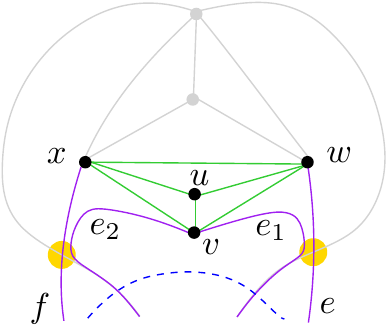}
    \label{fig:degree3_uncrossed_fig2}
  }
  \hfill\hfill
  \caption{Illustration of \cref{lem:degree_three_uncrossed_neighbors}
    where $\deg(v)=5$ and $\deg(w)=\deg(x)=4$. The vertices~$v$
    and~$u$ can be flipped to the other side of the edge~$wx$.}
  \label{fig:degree_three_uncrossed_fig}
\end{figure}

\section{Proofs from \cref{sec:charge}: The charging scheme}

\ChargeHermit*
\begin{proof}\label{PChargeHermit}
  Let~$h$ be a hermit in~$G$, and let~$uv$ be the edge of~$G$ that
  hosts~$h$. In the following we consider the vertex~$v$ only, the
  argument for~$u$ is symmetric.

  If~$\deg(v)\ge 5$, then~$h$ claims the two halfedges~$\half{hv}$
  and~$\half{uv}$ at~$v$. Clearly, the halfedge~$\half{hv}$ incident
  to~$h$ can be claimed by~$h$ only.  As for~$\half{uv}$, we first
  observe that no other hermit claims~$\half{uv}$ by
  \cref{lem:one_hermit_per_edge}. A T3-1 vertex claims only one
  uncrossed halfedge, which is incident to it, and both~$u$ and~$v$
  have degree at least four by
  \cref{lem:efficient_hermit_neighbors}. A T4-H vertex does not claim
  any uncrossed halfedge. So no T3-1 or T4-H vertex
  claims~$\half{uv}$. It remains to consider the peripheral edges of
  T3-2 and T3-3 vertices.

  So let~$z$ be a T3-2 or T3-3 vertex, and let~$\half{xw}$ be a
  peripheral halfedge that is assessed by~$z$. We may assume that the
  path~$xzw$ is uncrossed in~$D$, using \cref{prop:t32-choose} in
  case~$z$ is of type T3-2. Assume for a contradiction
  that~$\half{xw}=\half{uv}$. Then we redraw the uncrossed path~$uhv$
  to closely follow the uncrossed path~$xzw$, creating a quadrilateral
  face~$Q$ bounded by~$uhvz$ in the resulting $2$-plane drawing~$D'$
  of~$G$. Now we can add an uncrossed edge~$hz$ inside~$Q$, in
  contradiction to the maximality of~$G$. Thus, no vertex other
  than~$h$ claims~$\half{uv}$.
\end{proof}

\ClaimDouble*
\begin{proof}\label{PClaimDouble}
  Let~$h$ be a doubly crossed halfedge at~$v$ that is claimed by a
  vertex~$u$.

  If~$u$ is of type~T4-H, then by
  \cref{lem:degree4_hermit}, the first crossing~$\alpha$
  of~$h$ when traversing it starting from~$v$ is with a doubly crossed
  edge~$ux$ for which the arc~$u\alpha$ is uncrossed in~$D$. Given
  that~$ux$ is doubly crossed, the above statement determines~$u$
  uniquely. In particular, no other vertex claims~$h$.

  Consider the case that~$u$ is a T3-2 vertex. Let~$uw$
  denote the edge incident to~$u$ that is crossed by~$h$ in~$D$, and
  denote this crossing by~$\alpha$. Then by \cref{lem:cradj} the
  arc~$v\alpha$ of~$h$ is uncrossed in~$D$. Thus, the only other
  vertex that could claim~$h$ is~$w$. However, by
  \cref{lem:degree3_1crossed_degrees555} we have~$\deg(w)\ge 5$, and
  so~$w$ does not claim~$h$, either.
  
  It remains to consider the case that $u$ is of type~T3-1.
  Then the first crossing~$\alpha$ of~$h$ when 
  traversing it starting from~$v$ is with the edge $uw$ such that the 
  arc~$u\alpha$ is uncrossed by \cref{lem:degree3_2crossed_neighbordeg5}. 
  Suppose for contradiction that $h$ is claimed
  by another vertex. Then by \cref{lem:degree3_2crossed_neighbordeg5} 
  this vertex must be $w$, it must be a T3-1 vertex,
  and the arc $\alpha w$ must be uncrossed.
  Denote the second (other than $\alpha$)
  crossing of $h$ by $\beta$, and let edge $pq$ cross $h$ at $\beta$.
  For at least one endpoint of~$pq$, suppose without loss of generality
  it holds for~$p$, the arc~$p\beta$ is uncrossed in~$D$.
  Further $p$ must appear on a common face in $D$ with either $u$ or $w$,
  suppose without loss of generality it holds for~$u$.
  Then by \cref{lem:uncrossed_edge}, $p$ and $u$ are connected by
  an uncrossed edge in~$D$.
  As~$v$ is the only
  neighbor of~$u$ via an uncrossed edge in~$D$, it follows that~$p=v$,
  in contradiction to~$D$ being a simple drawing (the edges~$h$
  and~$pq=vq$ cross). 
\end{proof}

\begin{figure}[htbp]
  \centering
  \includegraphics[page=1]{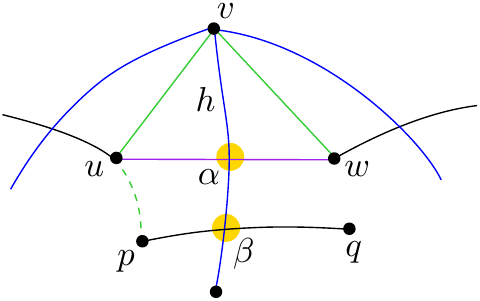}
  \caption{Proof of \cref{lem:double}: Two T3-1 vertices $u$ and $w$ claim the same halfedge $h$ at $v$ leading to a contradiction.}
  \label{fig:t3_1_multiclaims}
\end{figure}

\DegreeFive*
\begin{proof}\label{PDegreeFive}
  Consider a degree five vertex~$u$ in~$G$, and for a contradiction
  let~$v_1,v_2$ be two distinct low-degree vertices that claim
  halfedges at~$u$. We may suppose that~$\deg(v_1)\le\deg(v_2)$.

  By \cref{lem:degree4_hermit} a T4-H vertex claims halfedges at a
  vertex of degree at least six. So we have~$\deg(v_2)\le 3$. As each
  vertex of degree-three claims three halfedges at a vertex
  but~$\deg(u)=5$, we have~$\deg(v_1)=2$. Finally, by
  \cref{prop:hermit_deg} it follows that~$\deg(v_2)\ne 2$. So we
  have~$\deg(v_1)=2$ and~$\deg(v_2)=3$. Note that~$v_1$ claims two
  edges, both of which are uncrossed in~$D$: the edge~$uv_1$, which is
  incident to~$v_1$, and the edge that hosts~$v_1$. Also note
  that~$v_1$ and~$v_2$ are not adjacent in~$G$ by
  \cref{lem:efficient_hermit_neighbors}. In the following, we
  consider~$v_2$ in the role of the four different types of degree
  three vertices.

  \begin{case}[$v_2$ is a T3-3 vertex]
    Then the three halfedges claimed by~$v_2$ are uncrossed
    in~$D$. Together with the two uncrossed halfedges claimed by~$v_1$
    this means that all five edges incident to~$u$ are uncrossed
    in~$D$. But then we can add an edge~$v_1v_2$ to~$D$ by crossing
    only one, previously uncrossed edge, a contradiction to the
    maximality of~$G$.
  \end{case}
  
  \begin{case}[$v_2$ is a T3-2 vertex]
    Then two neighbors~$w,x$ of~$v_2$ are adjacent to~$u$ via
    uncrossed edges in~$D$. As~$v_2$ also claims a crossed edge
    incident to~$u$, it follows that at least one of the uncrossed
    edges~$ux,uw$ is claimed by~$v_1$. As~$v_1$ are~$v_2$ are not
    adjacent in~$G$, it follows that one of~$ux$ or~$uw$ hosts~$v_1$.
    Then by \cref{prop:t32-choose} we can arrange it so
    that~$v_1,y,v_2$ appear consecutively in the circular sequence of
    neighbors around~$u$ in~$D$, where~$y\in\{w,x\}$. As the edge~$uy$
    is uncrossed in~$D$, we can add an edge~$v_1v_2$ that crosses~$uy$
    only, a contradiction to the maximality of~$G$.
  \end{case}
  
  \begin{case}[$v_2$ is a T3-1 vertex]
    Then two of the halfedges claimed by~$v_2$ are crossed and the
    third one is~$v_2u$, which is uncrossed. The remaining two
    halfedges incident to~$u$ are the two uncrossed halfedges claimed
    by~$v_1$; see \cref{fig:degree5_hermit_t3}~(left). Then we can
    exchange the position of~$v_2$ and~$u$ in the drawing. As a
    result, all edges incident to~$u$ other than~$uv$ (which remains
    uncrossed) become singly crossed; see
    \cref{fig:degree5_hermit_t3}~(right). The resulting drawing is a
    $2$-plane drawing of~$G$ with strictly fewer doubly crossed edges
    than~$D$, a contradiction. (Alternatively, we could add an
    uncrossed edge~$v_1v_2$, in contradiction to the maximality
    of~$G$.)
  \end{case}

  We arrive at a contradiction in all cases. So, no such
  vertices~$v_1,v_2$ exist in~$D$. 
\end{proof}

\begin{figure}[htbp]
  \centering\includegraphics[page=1]{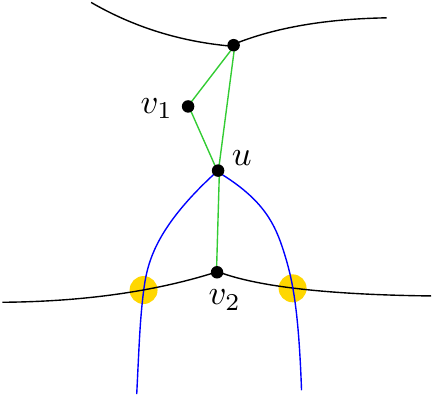}
  \qquad
  \includegraphics[page=1]{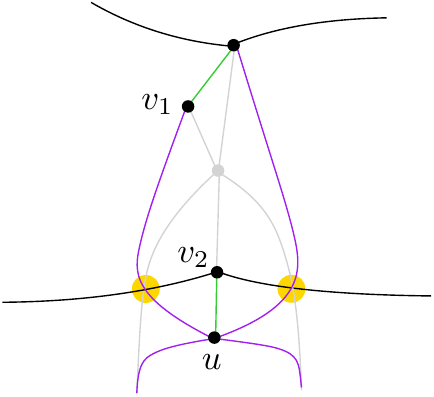}
  \caption{Proof of \cref{lem:deg5_atmost_one_low_deg}: A degree
    five vertex~$u$ adjacent to a hermit~$v_1$ and a T3-1
    vertex~$v_2$~(left); redraw to obtain a contradiction (right).}
  \label{fig:degree5_hermit_t3}
\end{figure}

\section{The Upper Bound: Proof of \cref{thm:upper}} \label{sec:upper_bound_full}

In this section we describe a construction for a family of maximal
$2$-planar graphs with few edges. The graphs can roughly be described
as braided cylindrical grids. More precisely, for a given~$k\in\N$ we
construct our graph~$G_k$ on~$10k+140$ vertices as follows.
\begin{itemize}
\item Take~$k$ copies of~$C_{10}$, the cycle on~$10$ vertices, and
  denote them by~$D_1,\ldots,D_k$. Denote the vertices of~$D_i$,
  for~$i\in\{1,\ldots,10\}$, by~$v_0^i,\ldots,v_9^i$ so that the edges
  of~$D_i$ are~$\{v_j^iv_{j\oplus 1}^i\colon 0\le j\le 9\}$,
  where~$\oplus$ denotes addition modulo~$10$.
\item For every~$i\in\{1,\ldots,k-1\}$, connect the vertices of~$D_i$ and~$D_{i+1}$ by a braided matching, as follows. For~$j$ even, add the edge~$v_j^iv_{j\oplus 8}^{i+1}$ to~$G_K$ and for~$j$ odd, add the edge~$v_j^iv_{j\oplus 2}^{i+1}$ to~$G_K$. See \cref{fig:nested_decagons}~(left) for illustration.
\item To each edge of~$D_1$ and~$D_k$ we attach a
  gadget~$X\simeq K_9\setminus(K_2+K_2+P_3)$ so as to forbid crossings
  along these edges. Denote the vertices of~$X$ by~$x_0,\ldots,x_8$
  such that~$\deg_X(x_0)=\deg_X(x_1)=8$, $\deg_X(x_8)=6$ and all other
  vertices have degree seven. Let~$x_6,x_7$ be the non-neighbors
  of~$x_8$. To an edge~$e$ of~$D_1$ and~$D_k$ we attach a copy of~$X$ so
  that~$e$ takes the role of the edge~$x_6x_7$ in this copy of~$X$. As
  altogether there are~$20$ edges in~$D_1$ and~$D_k$ and each copy
  of~$X$ adds seven more vertices, a total of~$20\cdot 7=140$ vertices
  are added to~$G_K$ with these gadgets.
\item For both~$D_1$ and~$D_k$ we connect their vertices by ten more
  edges each, essentially adding all edges of length two along each
  cycle. More explicitly, we add the edges~$v_j^iv_{j\oplus 2}^i$, for
  all~$0\le j\le 9$ and~$i\in\{1,k\}$.
\end{itemize}
This completes the description of the graph~$G_k$. Note that~$G_k$
has~$10k+140$ vertices and~$10k+10(k-1)+20\cdot 31+2\cdot 10=20k+630$
edges. So to prove \cref{thm:upper} asymptotically it suffices to
choose~$c\ge 630-2\cdot 140=350$ and show that~$G_k$ is maximal
$2$-planar. We will eventually get back to how to obtain the statement
for all values of~$n$.

To show that~$G_k$ is $2$-planar it suffices to give a $2$-plane
drawing of it. Such a drawing is given in \cref{fig:nested_decagons}:
(1)~We can simply nest the cycles~$D_1,\ldots,D_k$ with their
connecting edges using the drawing depicted in
\cref{fig:nested_decagons}~(left); and (2)~draw all copies of~$X$
attached to the edges of~$D_1$ and~$D_k$ using the drawing depicted in
\cref{fig:nested_decagons}~(right).

\begin{figure}[htbp]
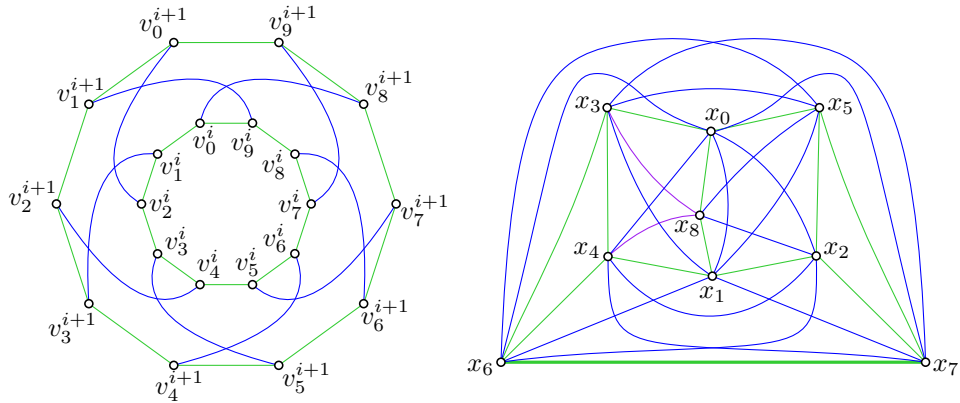

  \hfill
  \begin{minipage}[t]{.48\linewidth}
    \includegraphics[page=1]{decagons}
  \end{minipage}
  \hfill
  \begin{minipage}[t]{.48\linewidth}
    \includegraphics[page=2]{decagons}
  \end{minipage}
  \hfill
  \caption{The braided matching between two consecutive ten-cycles in~$G_k$, shown in blue (left); the gadget graph~$X$ that we attach to the edges of the first and the last ten-cycle of~$G_k$ (right).\label{fig:nested_decagons}}
\end{figure}

It is much more challenging, though, to argue that~$G_k$ is
\emph{maximal} $2$-planar. In fact, we do not know of a direct
argument to establish this claim. Instead, we will argue that~$G_k$
admits essentially only one $2$-plane drawing, which is the one
described above. Then maximality follows by just inspecting this
drawing and observing that no edge can be added there because every
pair of non-adjacent vertices is separated by a cycle of
doubly-crossed edges.

By \cref{thm:2planar_simple} we can restrict our attention to simple
drawings when arguing about maximality.

\begin{lemma}
  If a $2$-planar graph~$G$ is not maximal $2$-planar, then there exists a simple $2$-plane drawing of~$G$ that is not maximal $2$-plane.
\end{lemma}
\begin{proof}
  If~$G$ is not maximal $2$-planar, then there exists a pair~$u,v$ of nonadjacent vertices in~$G$ such that~$G'=G\cup e$ is $2$-planar. By \cref{thm:2planar_simple} there exists a simple $2$-plane drawing of~$G'$ whose restriction to~$G$ is a simple $2$-plane drawing of~$G$ that is not maximal $2$-plane.
\end{proof}

As a first step we show that~$X$ essentially has a unique $2$-plane
drawing. To prove this statement we need to make the term
\emph{essentially} more precise. As we consider~$X$ as a labeled
graph, we need to consider the automorphisms of~$X$. Recall
that~$X\simeq K_9\setminus(K_2+K_2+P_3)$. So there are four
equivalence classes of vertices: (1)~The unique vertex~$x_8$ of degree
six; (2)~the two non-neighbors of~$x_8$, which are~$x_6$ and~$x_7$;
(3)~the two vertices of degree eight, which are~$x_0$ and~$x_1$; and
(4)~the remaining vertices, all of degree seven. Both the vertices
in~(2) and~(3) can be freely and independently exchanged. And among
the vertices in~(4) there are two pairs of non-neighbors; we can
exchange both the pairs as a whole and also each pair
individually. Therefore, we get~$\mathrm{Aut}(X)=(Z_2)^5$
and~$|\mathrm{Aut}(X)|=32$ variations of the drawing depicted in
\cref{fig:nested_decagons}~(right). In all of these variations, there
is exactly one face that has~$x_6$ and~$x_7$ on its boundary (the
outer face in \cref{fig:nested_decagons}), and the uncrossed
edge~$x_6x_7$ is enclosed by a fourcycle of doubly crossed
edges. Therefore, the remaining vertices of~$X$ cannot interact with
(specifically receive additional edges from) the remainder
of~$G_k$. So we consider these~$32$ drawings to be essentially the
same. The claim is that there is no other relevant drawing of~$X$.

\begin{lemma}\label{lem:x}
  The graph~$X$ is maximal $2$-planar and it has an essentially unique simple $2$-plane drawing on the sphere~$\mathcal{S}^2$, the one depicted in \cref{fig:nested_decagons}~(right).  
\end{lemma}

We prove \cref{lem:x} using an exhaustive computational exploration of the space of simple drawings for~$X$. Our approach is similar to the one used by Angelini, Bekos, Kaufmann, and Schneck~\cite{AngeliniBKS20} for topological drawings of complete and complete bipartite graphs. The drawings are built incrementally and represented as a doubly-connected edge list (DCEL)~\cite[Chapter~2]{bkos-cgaa-08}. We insert a first edge to initialize, and then consider the remaining edges in some order such that whenever an edge is handled, at least one of its endpoints is in the drawing already (which is always possible if the graph is connected). When inserting an edge, we try (1)~all possible positions in the rotation at the source vertex and for each such position to (2)~draw the edge with zero, one or two crossings. Each such attempt amounts to a traversal of some face incident to the source vertex, and up to two more faces in the neighborhood (possibly a face appears several times in this process, as not every crossing necessarily changes the face and/or a crossing may return to a previously visited face).

Whenever we consider crossing an edge, we make sure to not cross an adjacent edge nor an edge that has been crossed already, so as to maintain a simple drawing; we also ensure that no edge is crossed more than twice, so as to maintain a $2$-plane drawing. If the target vertex is not yet present in the drawing, then we insert it into the corresponding face; otherwise, we have to find it on the boundary of the face. If the insertion of an edge is successful, we put the current state of the insertion search on a stack, and continue with the next edge. Otherwise, the edge could not be inserted and we backtrack and pop the previously inserted edge from the stack and proceed to go over the remaining options to insert this edge. 

Whenever all edges have been successfully added to the drawing, we output this drawing and backtrack to discover the possibly remaining drawings. Eventually, all options have been explored and we know the collection of possible drawings. As the algorithm considers all possible labeled drawings of the given graph, and the number of such drawings is typically exponential in the number of vertices, we can only use it for small, constant size graphs. Luckily our graph~$X$ described above is sufficiently small, with nine vertices and $32$~edges. We can verify that~$X$ is maximal $2$-planar either by studying the drawing shown in \cref{fig:nested_decagons}~(right), or by running the algorithm a few more times to obtain the following statement.

\begin{lemma}
  No graph that is obtained by removing (at most) three edges from~$K_9$ is $2$-planar.
\end{lemma}

When attempting to prove that~$G_k$ is maximal $2$-planar in a similar fashion, we face an obvious challenge: The graph~$G_k$ is definitely \emph{not} of small, constant size. In order to overcome this challenge, we (1)~leave out all copies of~$X$ from the graph (we know their drawing, anyway, by \cref{lem:x}) and replicate their role in~$G_k$ by initializing the edges of~$D_1$ and~$D_k$ to start with two fake crossings, effectively making them uncrossable for other edges and (2)~compute the possible drawings iteratively: Start with the unique drawing of~$D_1$, and iteratively consider all partial drawings obtained so far by adding another~$D_i$ and then clean up and remove those parts of the resulting drawings that are not useable anymore (for instance, because they are enclosed by doubly-crossed edges). The hope is that such a cleanup step allows us to maintain a ``relevant'' drawing of small, constant size, although conceptually the whole drawing can be huge. So by \emph{partial drawing} we refer to a $2$-plane drawing that has been obtained by starting from~$D_1$ and successfully adding a sequence~$D_2,\ldots,D_i$ of ten-cycles as they appear in~$G_i$, and after each such extension applying the cleanup step.

\paragraph{\textbf{Cutting down the search space.}} Recall that our overall goal is to draw~$G_k$. To obtain a drawing of~$G_k$, eventually a partial drawing must be completed by extending it with~$D_k$.  Therefore, in fact, we try to extend each partial drawing twice: first with a normal, unconstrained~$D_i$ and second---if the first extension is possible---with~$D_k$, whose edges are uncrossable. Another hope is that the number of partial drawings to consider remains manageable. In order to help with this, we consider all partial drawings up to isomorphisms that map~$D_{i-1}$ to itself while maintaining the parity of the indices (because only~$D_{i-1}$ has edges to~$D_i$ and the aforementioned isomorphisms maintain the set of these edges).

We also analyze every partial drawing that is discovered to see whether it is conceivable that it eventually can be completed by extending it with~$D_k$. If this is impossible, then the drawing under consideration is irrelevant for the purposes of drawing~$G_k$. A necessary criterion for such an extension to exist is the following.

\begin{lemma}\label{lem:curves}
  A partial drawing~$\Gamma$ of~$G_i$, for some~$i\in\N$, can be extended to a $2$-plane drawing of~$G_k$, for some~$k>i$, only if there exists a face~$f$ in~$\Gamma$ and a collection~$\Phi=\{\phi_v\colon v\in V(D_i)\}$ of curves such that~$\phi_v$ connects~$v$ to (some point in the interior of)~$f$ in~$\Gamma$ and in the drawing~$\Gamma\cup\Phi$ every edge of~$\Gamma$ has at most two crossings.
\end{lemma}
\begin{proof}
  To be extensible to a drawing of~$G_k$, the partial drawing~$\Gamma$ must eventually be completed by~$D_k$. As the edges of~$D_k$ are uncrossable, all of~$D_k$ lies in a single face of~$\Gamma$. Each vertex~$v$ of~$D_i$ has a path (via~$D_{i+1},\ldots,D_{k-1}$) to a vertex of~$D_k$. This path corresponds to a curve~$\phi_v$ as stated above. As~$\phi_v$ corresponds to a path in~$G_k$, we cannot say much about the number of crossings along~$\phi_v$. But the edges of~$\Gamma$ definitely must not be crossed more than twice overall.
\end{proof}

The existence of a face~$f$ and a set~$\Phi$ of curves as in \cref{lem:curves} can be decided efficiently by applying standard network flow techniques to the dual of~$\Gamma$: A single source~$s$ connects by a unit capacity edge to every vertex of~$D_i$, each vertex of~$D_i$ connects to all its incident faces, and every edge~$e$ of~$\Gamma$ connects the two incident faces by an edge of capacity two minus the number of crossings of~$e$ in~$\Gamma$. Then the question is whether or not there exists a flow of value ten from~$s$ to~$f$. We call every face of~$\Gamma$ for which such a flow exists a \emph{potential final face} of~$\Gamma$. By \cref{lem:curves} every partial drawing that does not have any potential final face can be discarded.

\paragraph{\textbf{Cleanup.}} Let us describe in a bit more detail the cleanup step that is applied to all partial drawings that the algorithm discovers. The cleanup is based on a classification of the faces. A face is \emph{active} if it has a flow of value at least three to some potential final face.

\begin{lemma}\label{lem:active}
  When extending a $2$-plane drawing~$\Gamma$ of~$G_i$ to a $2$-plane drawing of~$G_k$, for~$k>i$, all vertices of~$G_k\setminus G_i$ must be placed into active faces of~$\Gamma$.
\end{lemma}
\begin{proof}
  Consider a face~$f$ of~$\Gamma$, and suppose that in an extension~$\Gamma'$ of~$\Gamma$ to a drawing of~$G_k$, at least one vertex of~$G_j\setminus G_i$, for some~$i<j\le k$, is placed into~$f$. We have to show that~$f$ is active in~$\Gamma$. If~$\ell$ vertices of~$D_j$ are placed into~$f$ in~$\Gamma'$, then by \cref{lem:curves} there is a flow of value at least~$\ell$ to a potential final face of~$\Gamma$. So it remains to consider the case that only one or two vertices of~$D_j$ are placed into~$f$ in~$\Gamma'$. Let~$v$ be such a vertex. Then, as~$D_j$ forms a cycle in~$G_j$, two of its edges cross the boundary of~$f$ to reach vertices~$u,w$ outside of~$f$. By \cref{lem:curves} there is a set~$\Phi$ of three curves~$\phi_u,\phi_v,\phi_w$ from~$u,v,w$, respectively, to a potential final face of~$\Gamma$, such that every edge of~$\Gamma$ is crossed at most twice  in~$\Gamma\cup\Phi$. We can extend the curves~$\phi_u$ and~$\phi_w$ with the two paths from~$v$ to~$u,w$ along~$D_j$, respectively, to obtain a set~$\Phi'$ of three curves from~$v$ to a potential final face in~$\Gamma$ such that every edge of~$\Gamma$ is crossed at most twice in~$\Gamma\cup\Phi'$. 
  So~$\Phi'$ certifies the existence of a flow of value at least three from~$f$ to a potential final face, as claimed.
\end{proof}

In addition to active faces we also have passive and transit faces. A face~$f$ of a partial drawing is \emph{passive} if it is not active, it has a vertex~$v$ of~$D_i$ on its boundary, and in the dual there is a path~$P$ of length at most two from~$f$ to some active face such that the first edge of~$P$ is not incident to~$v$. A face of a partial drawing is \emph{transit} if it is neither active nor passive but in the dual it has at least one edge to an active face and at least one other edge that leads to an active or passive face (possibly two edges to the same active face). A face that is active, passive, or transit is called \emph{relevant}; all other faces are \emph{irrelevant}. By the following lemma, we can discard irrelevant faces from consideration for the purposes of drawing~$G_k$.

\begin{lemma}\label{lem:irrelevant}
  When extending a $2$-plane drawing~$\Gamma$ of~$G_i$ to a $2$-plane drawing of~$G_k$, for~$k>i$, the vertices and edges of~$G_k\setminus G_i$ are disjoint from the irrelevant faces of~$\Gamma$.
\end{lemma}
\begin{proof}
  By \cref{lem:active} all vertices of~$G_k\setminus G_i$ lie in active faces of~$\Gamma$. Every edge in~$G_k$ has at most two crossings, so it intersects at most three faces of~$\Gamma$. There are two types of edges in~$G_k\setminus G_i$: (1)~edges that connect two vertices of~$G_k\setminus G_i$ and (2)~edges that connect a vertex of~$D_i$ to a vertex of~$G_k\setminus G_i$. An edge of the first type starts and ends in an active face (possibly the same face). If it passes through another face in between, then this face satisfies the criterion for a transit face, which is, therefore, not irrelevant. An edge of the second type ends in an active face and starts in an active or passive face (possibly the same face where it ends). If it passes through another face in between, then this face satisfies the criterion for a transit face, which is, therefore, not irrelevant. 
\end{proof}

By \cref{lem:irrelevant} it is safe to remove irrelevant faces from
the drawing at the end of each iteration. In a similar fashion we also
remove irrelevant vertices. A vertex or crossing~$v$ of a drawing
of~$G_i$ is \emph{relevant} if it is (1)~a vertex of~$D_i$;
(2)~incident to two or more relevant faces; or (3)~incident to exactly
one relevant face that would become a lens (bounded by a two-cycle
only) if~$v$ was removed (in this case removing a vertex amounts to
contracting one of the two incident edges along the boundary of the
relevant face). All other vertices are \emph{irrelevant}. Finally, we
complete the cleanup step by adding an uncrossable star in each region
that has been evacuated by removing the irrelevant faces. This ensures
that the drawing is maintained and no edges can cross the
space---which now appears to be empty, although it is not---in future
iterations.

This completes the description of our computational approach to prove
that the graph~$G_k$ (essentially) has a unique $2$-plane drawing,
which proves \cref{thm:upper} asymptotically. It remains to argue how
to obtain the statement for all values of~$n$.

Technically, the smallest sensible choice for~$k$ is~$k=2$, which
leaves us with~$160$ vertices. To address the range~$n<160$, we use
that every $2$-planar graph on~$n$ vertices has at most~$5n-10$ edges
and simply set~$c\ge(5\cdot 159-10)-2\cdot 159=467$. By increasing~$k$
we only get values of~$n$ that are divisible by~$10$. To obtain all
other values in between, note that we can add a new vertex~$x'$ to a
copy of the gadget~$X$ and connect it to the three vertices~$x_3$,
$x_4$ and~$x_6$. Considering the (essentially) unique $2$-plane
drawing of~$X$ (refer to \cref{fig:nested_decagons}), in any $2$-plane
drawing of~$X^+:=X+x'$, the vertex~$x'$ must be placed into the one of
the four faces around the uncrossed triangle~$x_3x_4x_6$, and no other
edge incident to~$x'$ can be added to the graph while keeping it
$2$-planar. Therefore, by adding such a vertex to the desired number
of one to nine copies of~$X$ in~$G_k$, we can reach the full range of
values for~$n$. Each additional vertex adds three edges, rather than
two. But due to our generous setting of~$c$, the at most nine extra
edges are accounted for already.

\end{document}